\documentclass[11pt]{amsart}
\usepackage{amsmath}
\usepackage{amssymb}
\usepackage{amsfonts}
\usepackage{amsthm}  
\usepackage{latexsym}
\usepackage{graphicx}    
\usepackage{xypic} 
\usepackage{tikz} 
\usepackage{color} 
\usepackage{mathdots}
\usepackage{todonotes} 
\usepackage{mathrsfs} 
\usepackage{dutchcal} 
\usepackage{verbatim} 
\usepackage{pdfpages} 
\usepackage{mathtools} 
\usepackage[pdftex,linktocpage=true,pdfborder=true]{hyperref}
\usepackage[nameinlink]{cleveref} 

\crefname{thm}{theorem}{theorems}
\crefname{lem}{lemma}{lemmas}
\crefname{cor}{corollary}{corollaries}
\crefname{prop}{proposition}{propositions}
\crefname{defn}{definition}{definitions}
\crefname{eg}{example}{examples}
\crefname{xca}{exercise}{exercises}
\crefname{conj}{conjecture}{conjectures}
\crefname{rmk}{remark}{remarks}
\crefname{qst}{question}{questions}
\crefname{obs}{observation}{observations}

\newtheorem{thm}{Theorem}[section]
\newtheorem*{thm*}{Theorem} 
\newtheorem{lem}[thm]{Lemma}  

\newtheorem{prop}[thm]{Proposition}
\theoremstyle{definition}
\newtheorem{defn}[thm]{Definition}

\newtheorem{conj}[thm]{Conjecture}
\theoremstyle{remark}
\newtheorem{rmk}[thm]{Remark}

\newtheorem{obs}[thm]{Observation}
\newtheorem*{note}{Note} 
\numberwithin{equation}{section}

\DeclareMathAlphabet{\matholdcal}{OMS}{cmsy}{m}{n}

\newcommand{\ip}[2]{\langle #1 , #2 \rangle}    
\newcommand{\abs}[1]{\lvert#1\rvert}
\newcommand\diag{\operatorname{diag}}   
\newcommand\C{\mathbb C}    
\newcommand\R{\mathbb R}    
\newcommand\Z{\mathbb Z}    

\newcommand\style{\matholdcal}
\newcommand\tran{{}^t} 
\newcommand\st{\operatorname{St}} 
\newcommand\GL{\mathbf{GL}} 
\newcommand\gl{\mathfrak{gl}} 
\newcommand\SL{\mathrm{SL}} 
\newcommand\G{\mathbf{G}} 
\newcommand\Sp{\mathbf{Sp}} 
\newcommand\Hbf{\mathbf{H}} 
\newcommand\X{\mathbf{X}} 
\newcommand\Abf{\mathbf{A}} 
\newcommand\Mbf{\mathbf{M}} 
\newcommand\of{\mathcal O_F} 
\newcommand\wf{\matholdcal W_F} 
\newcommand\lf{\matholdcal L_F} 
\newcommand\Exp{\mathcal{Exp}} 
\newcommand\wt{\widetilde} 
\newcommand\ve{\varepsilon} 

\DeclareMathOperator{\spn}{span} 
\DeclareMathOperator{\id}{Id} 
\DeclareMathOperator{\Ad}{Ad}  
\DeclareMathOperator{\ind}{Ind} 
\DeclareMathOperator{\Hom}{Hom} 
\DeclareMathOperator{\Sym}{Sym} 

\begin{document}
\title{Speh representations are relatively discrete}
\author[J.~M.~Smith]{Jerrod Manford Smith}
\address{Department of Mathematics \& Statistics, University of Calgary, Calgary, Alberta, Canada, T2N 1N4}
\email{jerrod.smith@ucalgary.ca}
\urladdr{}
\thanks{}

\subjclass[2010]{Primary 22E50; Secondary 22E35}
\keywords{Relative discrete series, distinguished representation, symplectic group, Speh representation}
\date{\today}
\dedicatory{}
\begin{abstract}
Let $F$ be a $p$-adic field of characteristic zero and odd residual characteristic.
Let $\mathbf{Sp}_{2n}(F)$ denote the symplectic group defined over $F$, where $n\geq 2$.
We prove that the Speh representations $\matholdcal{U}(\delta,2)$, where $\delta$ is a discrete series representation of $\GL_n(F)$, lie in the discrete spectrum of the $p$-adic symmetric space $\mathbf{Sp}_{2n}(F) \backslash \mathbf{GL}_{2n}(F)$.  
\end{abstract}
\maketitle
\setcounter{tocdepth}{1}
\tableofcontents
\section{Introduction}
Let $F$ be a nonarchimedean local field of characteristic zero and odd residual characteristic $p$.
Let $G = \GL_{2n}(F)$ be an even rank general linear group and let $H = \Sp_{2n}(F)$ be the symplectic group.
This paper is concerned with the harmonic analysis on the $p$-adic symmetric space $X = H\backslash G$.
We prove that the Speh representations $\style{U}(\delta,2)$ appear in the discrete spectrum of $X$, as predicted by the conjectures of Sakellaridis and Venkatesh \cite{sakellaridis--venkatesh2017}.
Our main result, \Cref{thm-speh-rds}, is an unpublished result of Jacquet.  
We frame this result within the construction of (relative) discrete series representations for symmetric quotients of general linear groups carried out in \cite{smith2018,smith2018b}.
The present work relies on the substantial contributions of Heumos and Rallis \cite{heumos--rallis1990}, and Offen and Sayag \cite{offen--sayag2007,offen--sayag2008a,offen--sayag2008b} in the study of symplectic periods for the general linear group.

All representations are assumed to be on complex vector spaces.  In general, a smooth representation $(\pi,V)$ of $G$ is relevant to the harmonic analysis on $X=H\backslash G$ if and only if there exists a nonzero $H$-invariant linear form on the space $V$. If there exists a nonzero element $\lambda$ of $\Hom_H(\pi,1)$, then $(\pi,V)$ is $H$-distinguished.
Let $(\pi,V)$ be an irreducible admissible representation of $\GL_{2n}(F)$.
 Heumos and Rallis proved that the dimension of the space of $\Sp_{2n}(F)$-invariant linear forms on $V$ is at most one \cite[Theorem 2.4.2]{heumos--rallis1990}.
In addition, Heumos and Rallis showed that any irreducible admissible representation of $\GL_{2n}(F)$ cannot be both generic and $\Sp_{2n}(F)$-distinguished.
Recall that representation of $\GL_n(F)$ is generic if it admits a Whittaker model (see \cite{rodier1973} for more information on Whittaker models).

To see that an $H$-distinguished smooth representation $(\pi,V)$ of $G$ occurs in the space $C^\infty(X)$ of smooth (locally constant) functions on $X=H\backslash G$ one considers its relative matrix coefficients.
Let $\lambda \in \Hom_H(\pi,1)$ be nonzero.  For any $v\in V$, define a function $\varphi_{\lambda,v}$ by declaring that $\varphi_{\lambda,v}(Hg) = \ip{\lambda}{\pi(g)v}$.  The functions $\varphi_{\lambda,v}$ are smooth, since $\pi$ is smooth, and well-defined because $\lambda$ is $H$-invariant.  Moreover, the map that sends $v\in V$ to the $\lambda$-relative matrix coefficient $\varphi_{\lambda,v}$ intertwines $(\pi,V)$ and the right regular representation of $G$ on $C^\infty(X)$.
It is a fundamental problem to determine which irreducible representations of $G$ occur in the space $L^2(X)$ of square integrable functions on $X$.  The discrete spectrum $L^2_{\operatorname{disc}}(X)$ of $X$ is the direct sum of all irreducible $G$-subrepresentations of the space $L^2(X)$ of square integrable functions on $X$.
We prove, in \Cref{thm-speh-rds}, that the Speh representations $\style{U}(\delta,2)$ appear in $L^2_{\operatorname{disc}}(\Sp_{2n}(F)\backslash \GL_{2n}(F))$.
On the other hand, we do not prove that such representations are the only discrete series; we face the same obstacles discussed in \cite[Remark 6.6]{smith2018}.

Sakellaridis and Venkatesh have developed a framework ecnompassing the study of harmonic analysis on $p$-adic symmetric spaces and its deep connections with periods of automorphic forms and Langlands functoriality \cite{sakellaridis--venkatesh2017}.
In addition to providing explicit Plancherel formulas, Sakellaridis and Venkatesh have made precise conjectures describing the Arthur parameters of representations in the discrete series of symmetric spaces (and, more generally, spherical varieties) \cite[Conjectures 1.3.1 and 16.2.2]{sakellaridis--venkatesh2017}.
In fact, their conjectures predict that the discrete series of $\Sp_{2n}(F)\backslash \GL_{2n}(F)$ consists precisely of the Speh representations. 

We conclude the introduction with a summary of the contents of the paper.  In \Cref{sec-notation} we establish notation regarding $p$-adic symmetric spaces and representations; in addition, we review the Relative Casselman Criterion established by Kato and Takano \cite{kato--takano2010}.
We review the construction of the Speh representation in \Cref{sec-speh}.  In \Cref{sec-parameters}, we review the conjectures of Sakellaridis and Venkatesh and we demonstrate that their work predicts that the Speh representations $\style{U}(\delta,2)$ should appear in the discrete spectrum of $\Sp_{2n}(F)\backslash \GL_{2n}(F)$ (see \Cref{prop-sp-parameters}).
We determine the fine structure of the symmetric space $\Sp_{2n}(F)\backslash \GL_{2n}(F)$ in \Cref{sec-tori-pblc}; in particular, we realize the group $\Sp_{2n}(F)$ as the fixed points of an involution $\theta$ on $\GL_{2n}(F)$, determine the restricted root system and maximal $\theta$-split parabolic subgroups of $\GL_{2n}(F)$ relative to $\theta$. 
In \Cref{sec-rel-cass} we prove our main result, \Cref{thm-speh-rds}, by applying the Relative Casselman Criterion (see \Cref{thm-relative-casselman}).

In \Cref{sec-elliptic-context}, we make an effort to set the present work within the program started in \cite{smith2018,smith2018b}, where relative discrete series representations have been systematically constructed via parabolic induction from distinguished discrete series representations of $\theta$-elliptic Levi subgroups.
In fact, we realize the Speh representations as quotients of representations induced from distinguished discrete series of certain maximal $\theta$-elliptic Levi subgroups.
The present setting is complicated by the fact representations induced from discrete series are generic and therefore not distinguished by the symplectic group.
In particular, although we expect that the construction of relative discrete series carried out in \cite{smith2018,smith2018b} should generalize, some care must be taken to handle ``disjointness-of-models" phenomena as in the case of the Whittaker and symplectic models \cite[Theorem 3.2.2]{heumos--rallis1990}, and Klyachko models \cite{offen--sayag2008b}.

\subsection*{Acknowledgements}
The author would like to thank Omer Offen and Yiannis Sakellaridis for many helpful discussions.  The author also thanks the anonymous referee for several helpful suggestions.

\section{Notation and terminology}\label{sec-notation}
Let $F$ be a nonarchimedean local field of characteristic zero and odd residual characteristic $p$.
Let $\of$ be the ring of integers of $F$.
 Fix a uniformizer $\varpi$ of $F$. 
 Let $q$ be the cardinality of the residue field $k_F$ of $F$. 
 Let $|\cdot |_F$ denote the normalized absolute value on $F$ such that $|\varpi|_F = q^{-1}$.  We reserve the notation $|\cdot|$ for the usual absolute value on $\C$.
 
\subsection{Reductive groups and $p$-adic symmetric spaces}
 Let $\G$ be a connected reductive group defined over $F$.  
Let $\theta$ be an $F$-involution of $\G$.  
Let $\Hbf = \G^\theta$ be the subgroup of $\theta$-fixed points in $\G$.
Write $G = \G(F)$ for the group of $F$-points of $\G$.  
The quotient $H\backslash G$ is a $p$-adic symmetric space.
We will routinely abuse notation and identify an algebraic $F$-variety $\X$ with its $F$-points $X = \X(F)$.
When the distinction is to be made, we will use boldface to denote the algebraic variety and regular typeface for the set of $F$-points.

For an $F$-torus $\mathbf A \subset G$, let $A^1$ be the subgroup $\mathbf A(\of)$ of $\of$-points of $A = \mathbf A(F)$.
We use $Z_G$ to denote the centre of $G$ and $A_G$ to denote the $F$-split component of the centre of $G$.
Let $X^*(G) = X^*(\G)$ denote the group of $F$-rational characters of the algebraic group $\G$.
If $Y$ is a subset of a group $G$, then let $N_G(Y)$ denote the normalizer of $Y$ in $G$ and let $C_G(Y)$ denote the centralizer of $Y$ in $G$.
\subsubsection{Tori and root systems relative to involutions}\label{sec-tori-involution}
An element $g \in G$ is $\theta$-split if $\theta(g) = g^{-1}$.  An $F$-torus $S$ contained in $G$ is $(\theta,F)$-split if $S$ is $F$-split and every element of $S$ is $\theta$-split. 
 
Let $S_0$ be a maximal $(\theta,F)$-split torus of $G$.  Let $ A_0$ be a $\theta$-stable maximal $F$-split torus of $G$ that contains $S_0$ \cite[Lemma 4.5(iii)]{helminck--wang1993}. 
Let $\Phi_0 = \Phi(G,A_0)$ be the root system of $G$ with respect to $A_0$.
Let $W_0 = W(G, A_0) = N_{G}(A_0) / C_{G}(A_0)$ be the Weyl group of $G$ with respect to $A_0$.

The torus $A_0$ is $\theta$-stable, so there is an action of $\theta$ on the $F$-rational characters $X^*(A_0)$; moreover, $\Phi_0$ is a $\theta$-stable subset of $X^*(A_0)$.
Recall that a base of $\Phi_0$ determines a choice of positive roots $\Phi_0^+$.

\begin{defn}\label{defn-theta-base}
A base $\Delta_0$ of $\Phi_0$ is called a $\theta$-base if for every positive root $\alpha \in \Phi_0^+$ such that $\theta(\alpha) \neq \alpha$ we have that $\theta(\alpha) \in \Phi_0^-=-\Phi_0^+$.   
\end{defn}

Let $r: X^*(A_0) \rightarrow X^*(S_0)$ be the surjective map defined by restriction of ($F$-rational) characters.
Define $\overline \Phi_0 = r(\Phi_0) \setminus \{0\}$ and $\overline \Delta_0 = r(\Delta_0) \setminus \{0\}$.
The set $\overline \Phi_0$ coincides with $\Phi_0(G, S_0)$ and is referred to the as the restricted root system of $G/H$ \cite[Proposition 5.9]{helminck--wang1993}.
The set $\overline \Delta_0$ is a base of the root system $\overline \Phi_0$.  Note that $\overline \Phi_0$ is not necessarily reduced.
Let $\Phi_0^\theta$ and $\Delta_0^\theta$ be the subsets of $\theta$-fixed roots in $\Phi_0$, respectively $\Delta_0$.
Observe that $\overline \Phi_0 = r(\Phi_0 \setminus \Phi_0^\theta)$ and $\overline \Delta_0 = r(\Delta_0 \setminus \Delta_0^\theta)$.

Let $\overline \Theta$ be a subset of $\overline \Delta_0$.  Set $[\overline \Theta] = r^{-1}(\overline \Theta) \cup \Delta_0^\theta$.  Subsets of $\Delta_0$ of the form $[\overline \Theta]$ are called $\theta$-split.  Maximal $\theta$-split subsets of $\Delta_0$ are of the form $[\overline\Delta_0 \setminus \{\bar \alpha \}]$, where $\bar \alpha \in \overline \Delta_0$.

\subsubsection{Parabolic subgroups relative to involutions}\label{sec-pblc-involution}
Let $\mathbf{P}$ be an $F$-parabolic subgroup of $\G$.  We refer to an $F$-parabolic subgroup of $\G$ simply as a parabolic subgroup.
Let $\mathbf{N}$ be the unipotent radical of $\mathbf P$.  The reductive quotient $\mathbf M \cong \mathbf P / \mathbf N$ is called a Levi factor of $\mathbf P$.
We denote by $\delta_P$ the modular character of $P = \mathbf P(F)$ given by $\delta_P(p) = |\det \Ad_{\mathfrak n} (p) |_F$, where $\mathfrak n$ is the Lie algebra of $\mathbf N$.

Let $M$ be a Levi subgroup of $G$.  Let $A_M$ denote the $F$-split component of the centre of $M$.  The $(\theta,F)$-split component of $M$, denoted by $S_M$, is the largest $(\theta,F)$-split torus of $M$ that is contained in $A_M$.  
More precisely,
\begin{align*}
 S_M & = \left( \{ a \in A_M : \theta(a) = a^{-1} \} \right)^\circ,	
\end{align*}
where $(\cdot)^\circ$ denotes the Zariski-connected component of the identity.

\begin{defn}\label{defn-theta-split-pblc}
	A parabolic subgroup $P$ of $G$ is $\theta$-split if $\theta(P)$ is opposite to $P$, in which case $M = P \cap \theta(P)$ is a $\theta$-stable Levi subgroup of $P$.
\end{defn}

If $\Theta \subset \Delta_0$ is $\theta$-split, then the $\Delta_0$-standard parabolic subgroup $P_\Theta$ is $\theta$-split.
Let $\Phi_\Theta$ be the subsystem of $\Phi_0$ generated by $\Theta$.
The standard parabolic subgroup $P_\Theta$ has unipotent radical $N_\Theta$ generated by the root subgroups $N_\alpha$, where $\alpha \in \Phi_0^+ \setminus \Phi_\Theta^+$.
The standard Levi subgroup of $P_\Theta$ is $M_\Theta$, which is the centralizer in $G$ of the $F$-split torus $A_\Theta = \left( \bigcap_{\alpha \in \Theta} \ker \alpha \right)^\circ$.
Any $\Delta_0$-standard $\theta$-split parabolic subgroup of $G$ arises from a $\theta$-split subset of $\Delta_0$ \cite[Lemma 2.5(1)]{kato--takano2008}.

Let $\Theta \subset \Delta_0$ be $\theta$-split.
The $(\theta,F)$-split component of $M_\Theta$ is equal to
\begin{align*}
	S_\Theta &= \left( \bigcap_{\bar \alpha \in r(\Theta)} \ker (\bar\alpha : S_0 \rightarrow F^\times) \right)^\circ
\end{align*}
For any $0 < \epsilon \leq 1$, define 
\begin{align}\label{eq-split-dominant-part}
S_\Theta ^-(\epsilon) = \{ s \in S_\Theta  : |\alpha(s)|_F \leq \epsilon, \ \text{for all} \ \alpha \in \Delta_0 \setminus \Theta \}.
\end{align}
We write $S_\Theta ^-$ for $S_\Theta ^-(1)$ and refer to $S_\Theta ^-$ as the dominant part of $S_\Theta$.

By \cite[Theorem 2.9]{helminck--helminck1998}, the $\theta$-split subset $\Delta_0^\theta$ determines the standard minimal $\theta$-split parabolic subgroup $P_0 = P_{\Delta_0^\theta}$.
Let $N_0$ be the unipotent radical of $P_0$.  The standard Levi subgroup $M_0$ of $P_0$ is the centralizer in $G$ of the maximal $(\theta,F)$-split torus $S_0$.

\begin{lem}[{\cite[Lemma 2.5]{kato--takano2008}}]\label{KT08-lem-2.5}
Let $S_0 \subset A_0$, $\Delta_0$, and $P_0 = M_0N_0$ be as above.
\begin{enumerate}
\item Any $\theta$-split parabolic subgroup $P$ of $G$ is conjugate to a $\Delta_0$-standard $\theta$-split parabolic subgroup by an element $g \in (\Hbf \mathbf M_0)(F)$.
\item If the group of $F$-points of the product $(\Hbf \mathbf M_0)(F)$ is equal to $HM_0$, then any $\theta$-split parabolic subgroup of $G$ is $H$-conjugate to a $\Delta_0$-standard $\theta$-split parabolic subgroup.
\end{enumerate}
\end{lem} 

Let $P = MN$ be a $\theta$-split parabolic subgroup. Pick $g\in (\Hbf \mathbf M_0)(F)$ such that $P = gP_\Theta  g^{-1}$ for some $\theta$-split subset $\Theta \subset \Delta_0$.  
Since $g\in (\Hbf \mathbf M_0)(F)$ we have that $g^{-1}\theta(g) \in \mathbf M_0(F)$, and
we have $S_M = g S_\Theta g^{-1}$.
For a given $\epsilon >0$, one may extend the definition of $S_\Theta ^-$ in \eqref{eq-split-dominant-part}  to the torus $S_M$. 
Set $S_M^-(\epsilon) = g S_\Theta ^-(\epsilon) g^{-1}$ 
and define $S_M^- = S_M^-(1)$. 
Write $S_M^1$ to denote the group of $\of$-points $S_M(\of)$.

\subsection{Distinguished representations and relative matrix coefficients}
A representation $(\pi,V)$ of $G$ is smooth if for every $v\in V$ the stabilizer of $v$ in $G$ is an open subgroup.  
A smooth representation $(\pi,V)$ of $G$ is admissible if, for every compact open subgroup $K$ of $G$, the subspace $V^K$ of $K$-invariant vectors is finite dimensional. 
All of the representations that we consider are smooth and admissible. 
A quasi-character of $G$ is a one-dimensional representation.  
Let $(\pi,V)$ be a smooth representation of $G$. If $\omega$ is a quasi-character of $Z_G$, then $(\pi,V)$ is called an $\omega$-representation if $\pi$ has central character $\omega$.  

Let $P$ be a parabolic subgroup of $G$ with Levi subgroup $M$ and unipotent radical $N$.  
 Given a smooth representation $(\rho, V_\rho)$ of $M$ we may inflate $\rho$ to a representation of $P$, also denoted $\rho$, by declaring that $N$ acts trivially.
 We define the representation $\iota_P^G \rho$ of $G$ to be the (normalized) parabolically induced representation $\ind_P^G (\delta_P^{1/2} \otimes \rho)$. 
We will also use the Bernstein--Zelevinsky \cite{bernstein--zelevinsky1977,zelevinsky1980} notation $\pi_1 \times \ldots \times \pi_{k}$ for the (normalized) parabolically induced representation $\iota_{P_{(m_1,\ldots,m_k)}}^{\GL_m(F)}(\pi_1\otimes \ldots \otimes \pi_k)$ of $\GL_m(F)$ obtained from the standard (block-upper triangular) parabolic subgroup $P_{(m_1,\ldots,m_k)}$ and representations $\pi_{j}$ of $\GL_{m_j}(F)$, where $\sum_{j=1}^k m_j = m$.

Let $(\pi,V)$ be a smooth representation of $G$.  Let $(\pi_N, V_N)$ denote the normalized Jacquet module of $\pi$ along $P$.
Precisely,  $V_N$ is the quotient of $V$ by the $P$-stable subspace $V(N) = \spn\{ \pi(n)v - v : n\in N , v\in V\}$, and the action of $P$ on $V_N$ is normalized by $\delta_P^{-1/2}$. 
The unipotent radical $N$ of $P$ acts trivially on $(\pi_N,V_N)$ and we will regard $(\pi_N, V_N)$ as a representation of the Levi factor $M \cong P/ N$ of $P$.

Let $\pi$ be a smooth representation of $G$.  We also let $\pi$ denote its restriction to $H$.   Let $\chi$ be a quasi-character of $H$.  
\begin{defn}\label{defn-dist}
The representation $\pi$ is $(H,\chi)$-distinguished if the space $\Hom_H(\pi,\chi)$ is nonzero.  If $\pi$ is $(H,1)$-distinguished, where $1$ is the trivial character of $H$, then we will simply call $\pi$ $H$-distinguished.
\end{defn}

Let $(\pi,V)$ be a smooth $H$-distinguished $\omega$-representation of $G$.
Note that $\omega$ must be trivial on $Z_G\cap H$.
Let $\lambda \in \Hom_H(\pi,1)$ be a nonzero $H$-invariant linear functional on $V$.
Let $v\in V$ be a nonzero vector.  Define the $\lambda$-relative matrix coefficient associated to $v$ to be the complex valued function $\varphi_{\lambda,v}: G \rightarrow \C$ given by $\varphi_{\lambda,v}(g) = \ip{\lambda}{\pi(g)v}$.  When $\lambda$ is understood, we drop it from the terminology and refer to the relative matrix coefficients of $\pi$.
Since $(\pi,V)$ is assumed to be smooth, for all $v \in V$, the function $\varphi_{\lambda,v}$ lies in the space $C^\infty(G)$ of smooth (that is, locally constant) $\C$-valued functions on $G$.
Moreover, since $\pi$ is an $\omega$-representation, the functions $\varphi_{\lambda,v}$ lie in the subspace $C_\omega^\infty(G)$ consisting of smooth functions $f:G\rightarrow \C$ such that $f(zg) = \omega(z)f(g)$, for all $z\in Z_G$ and $g\in G$.
Observe that, since $\lambda$ is $H$-invariant, for all $g\in G, z\in Z_G$, and $h\in H$ we have
\begin{align*}
\varphi_{\lambda,v}(hzg) &= \ip{\lambda}{\pi(hzg)v} \\
& = \omega(z)\ip{\lambda}{\pi(g)v} \\
& = \omega(z)\varphi_{\lambda,v}(g).
\end{align*}
For any $v\in V$, the $\lambda$-relative matrix coefficient $\varphi_{\lambda,v}$ descends to well a defined function on $H\backslash G$ and satisfies $\varphi_{\lambda,v}(Hzg) = \omega(z) \varphi_{\lambda,v}(Hg)$, for $z\in Z_G$ and $Hg \in H\backslash G$.
Further assume that the central character $\omega$ of $(\pi,V)$ is unitary.
In this case, the function $Z_G H \cdot g \rightarrow \abs{\varphi_{\lambda,v}(g)}$ is well defined on $Z_G H \backslash G$.
The centre $Z_G$ of $G$ is unimodular since it is abelian.  
The fixed point subgroup $H$ is also reductive and thus unimodular. 
By \cite[Proposition 12.8]{Robert-book}, there exists a $G$-invariant measure on the quotient $Z_GH \backslash G$.

\begin{defn}
Let $\omega$ be a unitary character of $Z_G$.  Let $(\pi,V)$ be an $H$-distinguished $\omega$-representation of $G$.
Then $(\pi,V)$ is said to be
\begin{enumerate}
\item $(H,\lambda)$-relatively square integrable if and only if all of the $\lambda$-relative matrix coefficients are square integrable modulo $Z_GH$.  
\item $H$-relatively square integrable if and only if $\pi$ is $(H,\lambda)$-relatively square integrable for every $\lambda \in \Hom_H(\pi,1)$.
\end{enumerate}
\end{defn}

\subsection{Exponents and the Relative Casselman Criterion}
 Let $(\pi,V)$ be a finitely generated admissible representation of $G$.  
 Let $\Exp_{Z_G}(\pi)$ be the (finite) set of quasi-characters of $Z_G$ that occur as the central characters of the irreducible subquotients of $\pi$.  
 We refer to the characters that appear in $\Exp_{Z_G}(\pi)$ as the exponents of $\pi$.
 By \cite[Proposition 2.1.9]{Casselman-book}, the quasi-character $\chi:Z_G\rightarrow \C^\times$ occurs in $\Exp_{Z_G}(\pi)$ if and only if the generalized $\chi$-eigenspace for the action of $Z_G$ on $V$ is nonzero.
 Let $Z$ be a closed subgroup of $Z_G$.  The exponents $\Exp_{Z}(\pi)$ with respect to the action of $Z$ on $V$ are defined analogously.  
 If $Z_1 \supset Z_2$ are two closed subsets of $Z_G$, then the map $\Exp_{Z_1}(\pi) \rightarrow \Exp_{Z_2}(\pi)$ defined by restriction of quasi-characters is surjective (see, for instance, \cite[Lemma 4.15]{smith2018}).
 
 Let $P=MN$ be a parabolic subgroup of $G$ with unipotent radical $N$ and Levi factor $M$.  The normalized Jacquet module $(\pi_N,V_N)$ of $(\pi,V)$ along $P$ is also finitely generated and admissible \cite[Theorem 3.3.1]{Casselman-book}.  The set $\Exp_{A_M}(\pi_N)$ of exponents of $\pi_N$, with respect to the action of the $F$-split component $A_M$ of $M$, is referred to as the set of exponents of $\pi$ along $P$.
 \begin{lem}[{\cite[Lemma 4.16]{smith2018}}]\label{lem-induced-exp}
 Let $P=MN$ be a parabolic subgroup of $G$, let $(\rho,W)$ be a finitely generated admissible representation of $M$, and let $\pi = \iota_P^G\rho$. The exponents $\Exp_{A_G}(\pi)$ of $\pi$ are the restrictions to $A_G$ of the exponents $\Exp_{A_M}(\rho)$ of $\rho$.
 \end{lem}
 
 Let $(\pi,V)$ be a finitely generated admissible $H$-distinguished representation of $G$.  Let $\lambda \in \Hom_H(\pi,1)$ be a nonzero $H$-invariant linear form on $(\pi,V)$. In \cite{kato--takano2010}, Kato and Takano defined 
 \begin{align*}
 \Exp_{Z}(\pi,\lambda) & = \{ \chi \in \Exp_{Z}(\pi): \lambda\vert_{V_\chi} \neq 0\},	
 \end{align*}
 for any closed subgroup $Z$ of $Z_G$, where 
 \begin{align*}
 	V_\chi = \bigcup_{n=1}^\infty \{v \in V: (\pi(z)-\chi(z))^nv = 0, \forall z \in Z\}
 \end{align*}
is the generalized $\chi$-eigenspace for the $Z$ action on $V$.
 The elements of $\Exp_{A_G}(\pi,\lambda)$ are referred to as the exponents of $\pi$ relative to $\lambda$.
 Let $P$ be a $\theta$-split parabolic subgroup of $G$ with unipotent radical $N$ and $\theta$-stable Levi subgroup $M=P \cap \theta(P)$.
 Using Casselman's Canonical Lifting \cite[Proposition 4.1.4]{Casselman-book}, Kato--Takano \cite{kato--takano2008} and Lagier \cite{lagier2008} defined an $M^\theta$-invariant linear functional $\lambda_N \in \Hom_{M^\theta}(\pi_N,1)$, canonically associated to $\lambda$, on the Jacquet module $\pi_N$ of $\pi$ along $P$.
 We refer the reader to \cite[Proposition 5.6]{kato--takano2008} for details of the construction and additional properties of the map $\lambda \mapsto \lambda_N$.
We may now state the Relative Casselman Criterion \cite[Theorem 4.7]{kato--takano2010}.

\begin{thm}[Relative Casselman Criterion]\label{thm-relative-casselman}
	Let $\omega$ be a unitary character of $Z_G$. Let $(\pi,V)$ be a finitely generated admissible $H$-distinguished $\omega$-representation of $G$.  Fix a nonzero element $\lambda$ in $\Hom_H(\pi,V)$. 
	The representation $(\pi,V)$ is $(H,\lambda)$-relatively square integrable if and only if the condition
	\begin{align}\label{eq-rel-cass-condition}
		|\chi(s)|&<1, & \forall \chi \in \Exp_{S_M}(\pi_N,\lambda_N), \forall s \in S_M^-\setminus S_GS_M^1	
	\end{align}
	is satisfied for every proper $\theta$-split parabolic subgroup $P=MN$ of $G$.
\end{thm}

\subsection{Conventions regarding $\Sp_{2n}(F) \backslash \GL_{2n}(F)$}\label{sec-conventions}
From now on, unless otherwise specified, we let $G = \GL_{2n}(F)$ and let $H = \Sp_{2n}(F)$, where $n\geq 2$.   We will realize the symplectic group $H$ explicitly as the subgroup of $G$ fixed pointwise by the involution $\theta$ given by
\begin{align*}
	\theta(g) &= \ve_{2n}	^{-1} \tran{g}^{-1} \ve_{2n},
\end{align*}
where $\tran{g}$ denotes the transpose of $g\in G$,
\begin{align*}
\ve_{2n} & = \left(\begin{matrix} 0 & J_n \\ -J_n & 0 \end{matrix} \right) \in \GL_{2n}(F)	& \text{and} & & J_n & = \left(\begin{matrix} & & 1 \\ & \iddots & \\ 1 & & \end{matrix}\right) \in \GL_n(F).
\end{align*}
Note that $\ve_{2n}$ is a nonsingular skew-symmetric element of $G$; moreover, $\ve_{2n}^{-1} = \tran\ve_{2n}$.
With this choice of involution, the subgroup of upper triangular matrices in $H$ is a Borel subgroup (over $F$).
Let $A_0$ be the maximal diagonal $F$-split torus of $G$. 

There is a right $G$-action on the set of involutions of $G$ given by 
\begin{align}\label{inv-g-action}
g\cdot\theta(x) = g^{-1} \theta (gxg^{-1}) g,	
\end{align}
for any $x,g\in G$.
Any involution of the form $g\cdot \theta$ is said to be $G$-equivalent to $\theta$.
If $x\in G$ is skew-symmetric, then we obtain a realization of the symplectic group as the fixed points of the involution $\theta_x$ defined by
\begin{align*}
\theta_x(g) = x^{-1} \tran{g}^{-1} x
\end{align*}
Moreover, the $G$-action on involutions is compatible with the right $G$-action on the set of skew-symmetric matrices given by $x\cdot g = \tran{g}xg$, for any $g\in G$ and any skew-symmetric matrix $x\in G$.  
Indeed, if $x, y\in G$ and $x$ is skew-symmetric, then $y \cdot \theta_x = \theta_{x\cdot y}$.

We will write $\diag(a_1,\ldots, a_m)$ to denote the diagonal $m\times m$ matrix with entries $a_1, \ldots, a_m$ on the main  diagonal.  Given a partition $(m_1\ldots, m_k)$ of a positive integer $m$, write $P_{(m_1,\ldots,m_k)}$ for the standard block-upper triangular parabolic subgroup of $\GL_m(F)$ with Levi factor $M_{(m_1,\ldots,m_k)}$ and unipotent radical $N_{(m_1,\ldots,m_k)}$.
 Write $\nu$ for the unramified character $|\det (\cdot) |_F$ of $\GL_m(F)$, where $m$ is understood from context.

\section{Speh representations}\label{sec-speh}
Recall that a representation $\pi$ of $\GL_m(F)$ is said to be generic if it admits a Whittaker model, that is, if there exists a nonzero intertwining operator in the space $\Hom_{N_m}(\pi, \ind_{N_m}^{\GL_m(F)} \Psi_m)$, where $N_m$ is the subgroup of $\GL_m(F)$ consisting of upper triangular unipotent matrices and $\Psi_m$ is a non-degenerate character of $N_m$.

Let $\delta$ be an irreducible square integrable representation of $\GL_n(F)$.  The parabolically induced representation $\nu^{1/2}\delta \times \nu^{-1/2}\delta$ has length two and admits a unique irreducible generic subrepresentation $\style{Z}(\delta,2)$ and a unique irreducible quotient $\style{U}(\delta,2)$ \cite{bernstein--zelevinsky1977,zelevinsky1980}.
In particular, we have the following short exact sequence of $G$-representations
\begin{align}\label{eq-speh-exact}
0 \rightarrow \style{Z}(\delta,2) \rightarrow 	\nu^{1/2}\delta \times \nu^{-1/2}\delta \rightarrow \style{U}(\delta,2) \rightarrow 0.
\end{align}
The representations $\style{U}(\delta, 2)$ are the Speh representations.  

\begin{rmk}
	The Speh representations, and generalized Speh representations, feature prominently in the classification of the unitary dual of general linear groups carried out by Tadi\'{c} \cite{tadic1986}.
\end{rmk}

Heumos and Rallis proved that the Speh representations $\style{U}(\delta, 2)$ are $H$-distinguished by constructing a nonzero $H$-invariant linear functional on the full induced representation $\nu^{1/2}\delta \times \nu^{-1/2}\delta$ and then appealing to \cite[Theorem 3.2.2]{heumos--rallis1990} to show that the generic subrepresentation $\style{Z}(\delta,2)$ cannot be $H$-distinguished.  
One can then appeal to the exact sequence \eqref{eq-speh-exact} conclude that the invariant functional on $\nu^{1/2}\delta \times \nu^{-1/2}\delta$ descends to a well-defined nonzero $H$-invariant linear functional on $\style{U}(\delta,2)$.
Note that the invariant form on $\nu^{1/2}\delta\times\nu^{-1/2}\delta$ may be realized via \cite[Proposition 7.1]{offen2017}, or an explicit formulation of this result as in \cite[Lemma 4.10]{smith2018}.  

\begin{rmk}
The method used by Heumos and Rallis to demonstrate the $\Sp_{2n}(F)$-distinction of the Speh representations $\style{U}(\delta,2)$ does not immediately extend to the generalized Speh representations $\style{U}(\delta,m)$, $m>2$ (see \Cref{sec-sp-parameters} for the definition).  However, Offen and Sayag \cite{offen--sayag2007} study the distinction of the generalized Speh representations by utilizing work of Jacquet and Rallis \cite{jacquet--rallis1996} and Bernstein's meromorphic continuation.  The method of Offen and Sayag, used to prove the ``hereditary property of symplectic periods," is a special case of the method of Blanc and Delorme \cite{blanc--delorme2008}.  We refer the reader to \cite{offen--sayag2007} for more details.
\end{rmk}

Next, let us recall descriptions of the generalized Steinberg representations and generalized Speh representations of general linear groups.
Let $\rho$ be an irreducible unitary supercuspidal representation of $\GL_r(F)$, $r\geq 1$.  Let $k\geq2$ be an integer. By \cite[Theorem 9.3]{zelevinsky1980}, the induced representation 
\begin{align*}
\nu^{\frac{k-1}{2}} \rho \times \ldots \times \nu^{\frac{1-k}{2}}\rho	
\end{align*}
of $\GL_{kr}(F)$ admits a unique irreducible subrepresentation $\style{Z}(\rho,k)$; moreover, $\style{Z}(\rho,k)$ is square integrable.  Often in the literature, $\style{Z}(\rho,k)$ is denoted by $\st(\rho,k)$ and these are the generalized Steinberg representations of $\GL_{kr}(F)$.
Let $\delta$ be a discrete series representation of $\GL_{d}(F)$, $d \geq 2$.
Let $m\geq 2$ be an integer.
By \cite[Theorem 6.1(a)]{zelevinsky1980}, the induced representation
\begin{align*}
\nu^{\frac{m-1}{2}}\delta \times \ldots \times \nu^{\frac{1-m}{2}}\delta
\end{align*}
 admits a unique irreducible (unitary) quotient $\style{U}(\delta, m)$. 
 The representations $\style{U}(\delta,m)$ are the generalized Speh representations studied by Tadi\'c \cite{tadic1985a}.

\section{$X$-distinguished Arthur parameters}\label{sec-parameters}
In the following discussion, $\G$ can be taken to be an arbitrary connected reductive group that is split over $F$.  We return to $G = \GL_{2n}(F)$ in \Cref{sec-sp-parameters}.

Let $\style W_F$ be the Weil group of $F$ and let $\style{L}_F = \style{W}_F \times \SL(2,\C)$ be the Weil--Deligne group of $F$.
Since $\G$ is split over $F$, $\style W_F$ acts trivially on the complex dual group $G^\vee$, and the $L$-group of $\G$ can be identified with the dual group ${}^L G = G^\vee$.    
Recall that an Arthur parameter, or an $A$-parameter, for $G$ is a continuous homomorphism $\psi: \lf\times \SL(2,\C) \rightarrow G^\vee$ such that
 \begin{itemize}
 	\item  the restriction $\psi\vert_{\wf}$ of $\psi$ to the Weil group $\wf$ is bounded,
	\item  the image of $\psi\vert_{\wf}$ consists of semisimple elements of $G^\vee$,
 	\item  and the restriction of $\psi$ to each of the two $\SL(2,\C)$ factors is algebraic.
 \end{itemize}
A Langlands parameter, or an $L$-parameter, for $G$ is a continuous homomorphism $\phi: \lf \rightarrow G^\vee$ such that
\begin{itemize}
	\item the image of $\phi \vert_{\wf}$ consists of semisimple elements of $G^\vee$,
	\item and the restriction of $\phi$ to $\SL(2,\C)$ is algebraic.
\end{itemize}

Inspired by work of Gaitsgory and Nadler, Sakellaridis and Venkatesh have associated to any $\G$-spherical $F$-variety $\X$ a complex dual group $G_X^\vee$, see \cite[Sections 2--3]{sakellaridis--venkatesh2017} (provided that the assumption of \cite[Proposition 2.2.2]{sakellaridis--venkatesh2017} on the spherical roots is satisfied).  
In addition, they described a \textit{distinguished morphism} $\varrho: G_X^\vee \times \SL(2,\C) \rightarrow G^\vee$ satisfying certain properties and unique up to conjugation by a canonical maximal torus in $G^\vee$ \cite[Section 3.2]{sakellaridis--venkatesh2017}.
Existence of distinguished morphisms has been proven in full generality by Knop and Schalke \cite{knop--schalke2017}.

\begin{defn}
An $A$-parameter  $\psi: \lf \times \SL(2,\C) \rightarrow G^\vee$ is $X$-distinguished if it factors through the distinguished morphism $\varrho: G_X^\vee \times \SL(2,\C) \rightarrow G^\vee$, that is, if and only if there exists a tempered (that is, bounded on $\wf$) $L$-parameter $\psi_X: \lf \rightarrow G_X^\vee$ such that $\psi = \varrho \circ (\phi_X\otimes \id)$, where $\id: \SL(2,\C) \rightarrow \SL(2,\C)$ is the identity.
\end{defn}

\begin{defn}
An $X$-distinguished $A$-parameter is $X$-elliptic if it factors through $\varrho$ via an elliptic $L$-parameter $\phi_{X}:\style{L}_F \rightarrow  G_X^\vee$, that is, the image of $\phi_X$ is not contained in any proper Levi subgroup of  $ G_X^\vee$.
\end{defn}

We recall the following conjecture \cite[Conjecture 1.3.1]{sakellaridis--venkatesh2017}.
\begin{conj}[Sakellaridis--Venkatesh]\label{SV-conj-Arthur-supp}
The support of the Plancherel measure for $L^2(X)$, as a representation of $G$, is contained in the union of Arthur packets attached to $X$-distinguished $A$-parameters.
\end{conj}
In fact, Sakellaridis and Venkatesh give much more refined conjectures that predict a direct integral decomposition of $L^2(X)$ over $X$-distinguished $A$-parameters \cite[Conjecture 16.2.2]{sakellaridis--venkatesh2017}.
In addition, the refined conjectures make the following prediction about the $X$-distinguished $A$-parameters of the relative discrete series representations.
\begin{conj}[Sakellaridis--Venkatesh]\label{SV-discrete}A relative discrete series representation $\pi$ in $L^2(X)$ is contained in an Arthur packet corresponding to an $X$-distinguished $X$-elliptic $A$-parameter.
\end{conj}

Of course, in the setting we are concerned with, the situation is greatly simplified by the fact that Arthur packets (and $L$-packets) for the general linear group are singleton sets.

\subsection{Distinguished $A$-parameters for $\Sp_{2n}\backslash \GL_{2n}$}\label{sec-sp-parameters}
Let $\X$ be the symmetric variety $\Sp_{2n}\backslash \GL_{2n}$.  
The dual group of $G = \GL_{2n}(F)$ is $G^\vee = \mathrm{GL}(2n,\C)$.
The dual group $G_X^\vee$ of $\X$ is the rank-$n$ complex general linear group, that is, $G_X^\vee = \mathrm{GL}(n,\C)$ (see \Cref{lem-rest-roots-dual-gp}). 
 
 Let $\pi$ be an irreducible unitary $\Sp_{2n}(F)$-distinguished representation of $G$.  Let $\psi_\pi: \style{L}_F\times \SL(2,\C) \rightarrow \mathrm{GL}(2n,\C)$ be the $A$-parameter of $\pi$.
 We refer the reader to \cite{Arthur-book,xu2017a} for the description of the $A$-parameters of the representations in the unitary dual of $G$, including the generalized Steinberg and Speh representations. 
The distinguished morphism $\varrho : \mathrm{GL}(n,\C) \times \SL(2,\C) \rightarrow \mathrm{GL}(2n,\C)$ is given by the tensor product of the standard $n$-dimensional representation of $\mathrm{GL}(n,\C)$ with the standard $2$-dimensional representation $\style{S}(2)$ of $\SL(2,\C)$ \cite[Example 1.3.2]{sakellaridis--venkatesh2017}.  
Thus, if $\phi: \lf \rightarrow \mathrm{GL}(n,\C)$ is an $L$-parameter, then $\varrho \circ (\phi \otimes \id) = \phi \otimes \style{S}(2)$.
 By \Cref{SV-discrete}, we expect $\pi$ to be relatively discrete if its $A$-parameter $\psi_\pi$ has the following two properties.
\begin{description}
\item[P1]  The $A$-parameter $\psi_\pi$ is $X$-distinguished, that is, $\psi_\pi$ factors through the distinguished morphism $\varrho$; in particular, $\psi_\pi = \phi_{\pi,X} \otimes \style{S}(2)$, where $\phi_{\pi,X}: \style{L}_F \rightarrow \mathrm{GL}(n,\C)$ is a tempered $L$-parameter.
\item[P2]  The $L$-parameter $\phi_{\pi,X}: \style{L}_F \rightarrow \mathrm{GL}(n,\C)$ is elliptic, that is, the image of $\phi_{\pi,X}$ is not contained in any proper parabolic subgroup of $\mathrm{GL}(n,\C)$. 
\end{description}
Recall that, under the Local Langlands Correspondence, a tempered elliptic parameter $\phi: \style{L}_F \rightarrow \mathrm{GL}(n,\C)$ corresponds to a discrete series representation of $\GL_n(F)$.

\begin{prop}\label{prop-sp-parameters}
	Let $\pi$ be an irreducible unitary 	$\Sp_{2n}(F)$-distinguished  representations of $\GL_{2n}(F)$.  Let $\psi_\pi: \style{L}_F \times \SL(2,\C) \rightarrow \mathrm{GL}(2n,\C)$ be the $A$-parameter of $\pi$.
The $A$-parameter $\psi_\pi$ is $X$-distinguished and $X$-elliptic if and only if $\pi$ is isomorphic to a Speh representation $\style{U}(\delta,2)$ for some discrete series representation $\delta$ of $\GL_n(F)$.
\end{prop}

\begin{proof}	
	Let $\pi$ be an irreducible unitary 	$\Sp_{2n}(F)$-distinguished  representations of $\GL_{2n}(F)$.
	Following the notation of \cite{offen--sayag2007}, let $\pi(\sigma, \alpha) = \nu^{\alpha}\sigma \times \nu^{-\alpha}\sigma$, where $\alpha \in \R$ so that $|\alpha|<1/2$, and $\sigma$ is a smooth representation of $\GL_d(F)$, for some $d\geq 1$.
	Offen and Sayag \cite{offen--sayag2007,offen--sayag2008b} have shown that $\pi$ must be equivalent to a representation of the form
\begin{align*}
\left( \bigtimes_{i=1}^l \style{U}(\style{Z}(\rho_i, k_i), 2m_i) \right) 	\times \left( \bigtimes_{i=l+1}^t \pi(\style{U}(\style{Z}(\rho_i,k_i), 2m_i),\alpha_i)\right),
\end{align*}
where $2n = \sum_{i=1}^l 2k_ir_im_i + \sum_{i=l+1}^t 4k_ir_im_i$, the representation $\rho_i$ of $\GL_{r_i}(F)$ is irreducible, unitary and supercuspidal, and $\alpha_i \in \R$ with $|\alpha_i|<1/2$.

Let $\phi_{\rho_i}:\style{W}_F \rightarrow \mathrm{GL}(r_i,\C)$ be the $L$-parameter of the supercuspidal representation $\rho_i$.  Write  $\style{S}(d)\cong \Sym^{d-1}(\C^2)$ for the unique (up to isomorphism) $d$-dimensional irreducible representation of $\SL(2,\C)$.  
Let $|\cdot|_{\style{W}_F}:\style{W}_F \rightarrow \R_{>0}$ denote the absolute value on the Weil group given by $|\cdot|_{\style{W}_F} = |\operatorname{Art}_F^{-1}(\cdot)|_F$, where $\operatorname{Art}_F: F^\times \rightarrow \style{W}_F^{\operatorname{ab}}$ is the Artin map and $|\cdot|_F$ is the (normalized) absolute value on $F$.
Up to equivalence, the $A$-parameter $\psi_\pi$ of $\pi$ is equal to
\begin{align}\label{eq-all-sp-parameter}
\bigoplus_{i=1}^l \psi_{\style{U}(\style{Z}(\rho_i,k_i),2m_i)} \oplus \bigoplus_{i=l+1}^t \left(|\cdot|_{\style{W}_F}^{\alpha_i}\psi_{\style{U}(\style{Z}(\rho_i,k_i),2m_i)}	\oplus |\cdot|_{\style{W}_F}^{-\alpha_i}\psi_{\style{U}(\style{Z}(\rho_i,k_i),2m_i)} \right),
\end{align}
where
\begin{align}\label{eq-A-parameter-speh-type}
\psi_{\style{U}(\style{Z}(\rho_i,k_i),2m_i)} & 	= \phi_{\style{Z}(\rho_i,k_i)}\otimes \style{S}(2m_i) : \style{L}_F \times \SL(2,\C) \rightarrow \mathrm{GL}(2k_ir_im_i,\C)
\end{align}
is the $A$-parameter of the generalized Speh representation $\style{U}(\style{Z}(\rho_i,k_i),2m_i)$, and
\begin{align*}
\phi_{\style{Z}(\rho_i,k_i)} & 	= \phi_{\rho_i}\otimes \style{S}(k_i) : \style{L}_F  \rightarrow \mathrm{GL}(k_ir_i,\C)
\end{align*} 
is the $L$-parameter of the generalized Steinberg representation $\style{Z}(\rho_i,k_i)$.

The $A$-parameter $\psi_\pi$ is $X$-distinguished (\textbf{P1}) if and only if $\psi_\pi = \phi_{\pi,X} \otimes \style{S}(2)$, for some tempered $L$-parameter $\phi_{\pi,X}: \lf \rightarrow \mathrm{GL}(n,\C)$.
In light of \eqref{eq-all-sp-parameter} and \eqref{eq-A-parameter-speh-type}, we first notice that $\psi_\pi$ factors through the distinguished morphism $\varrho$ if and only if $m_i=1$, for all $1\leq i \leq t$.
In this case, 
\begin{align*}
\pi & \cong \left( \bigtimes_{i=1}^l \style{U}(\style{Z}(\rho_i, k_i), 2) \right) 	\times \left( \bigtimes_{i=l+1}^t \pi(\style{U}(\style{Z}(\rho_i,k_i), 2),\alpha_i)\right),
\end{align*}
and 
\begin{align*}
& \psi_\pi  = \bigoplus_{i=1}^l \psi_{\style{U}(\style{Z}(\rho_i,k_i),2)} \oplus \bigoplus_{i=l+1}^t \left(|\cdot|_{\style{W}_F}^{\alpha_i}\psi_{\style{U}(\style{Z}(\rho_i,k_i),2)}	\oplus |\cdot|_{\style{W}_F}^{-\alpha_i}\psi_{\style{U}(\style{Z}(\rho_i,k_i),2)} \right)\\
& = \left(\bigoplus_{i=1}^l \phi_{\style{Z}(\rho_i,k_i)} \oplus \bigoplus_{i=l+1}^t |\cdot|_{\style{W}_F}^{\alpha_i}\phi_{\style{Z}(\rho_i,k_i)}	\oplus |\cdot|_{\style{W}_F}^{-\alpha_i}\phi_{\style{Z}(\rho_i,k_i)} \right) \otimes\style{S}(2)\\
& = \phi_{\pi,X} \otimes \style{S}(2),
\end{align*}
where $\phi_{\pi,X}: \style{L}_F \rightarrow \mathrm{GL}(n,\C)$ is the $L$-parameter
\begin{align*}
\bigoplus_{i=1}^l \phi_{\style{Z}(\rho_i,k_i)} \oplus \bigoplus_{i=l+1}^t |\cdot|_{\style{W}_F}^{\alpha_i}\phi_{\style{Z}(\rho_i,k_i)}	\oplus |\cdot|_{\style{W}_F}^{-\alpha_i}\phi_{\style{Z}(\rho_i,k_i)}.	
\end{align*}
Moreover, the $L$-parameter $\phi_{\pi,X}$ is tempered if and only if $\phi_{\pi,X}$ restricted to $\wf$ has bounded image, that is, if and only if $l=t$ (or, equivalently, $\alpha_i=0$ for all $l+1 \leq i \leq t$). In particular, the representations $\pi(\style{U}(\style{Z}(\rho_i,k_i),2),\alpha_i)$ cannot appear in the inducing data of $\pi$.  Thus, $\psi_\pi$ is $X$-distinguished (\textbf{P1}) if and only if $\psi_\pi = \phi_{\pi,X} \otimes \style{S}(2)$, where
\begin{align}
\label{prop-temp-l} \phi_{\pi,X} & = \bigoplus_{i=1}^l \phi_{\style{Z}(\rho_i,k_i)}.
\end{align}
The $L$-parameter $\phi_{\pi,X}$ is $X$-elliptic (\textbf{P2}) if and only if and only if there is precisely one direct summand in \eqref{prop-temp-l}  (that is, $l=1$) and $\phi_{\pi,X} = \phi_{\style{Z}(\rho,k)}$ corresponds to the discrete series representation $\delta=\style{Z}(\rho,k)$ of $\GL_n(F)$.
In particular, $\psi_\pi$ is $X$-distinguished (\textbf{P1}) and $X$-elliptic (\textbf{P2}) if and only if  $\psi_\pi = \phi_{\style{Z}(\rho,k)} \otimes \style{S}(2)$, in which case, by the Local Langlands Correspondence, $\pi \cong \style{U}(\delta,2)$ is a Speh representation.
\end{proof}

In summary, \Cref{prop-sp-parameters} allows us to interpret \Cref{SV-discrete} as predicting that only the Speh representations $\style{U}(\delta,2)$, where $\delta=\style{Z}(\rho,k)$ is a discrete series representation of $\GL_n(F)$, appear in the discrete spectrum of $X = \Sp_{2n}(F) \backslash \GL_{2n}(F)$.
The goal of the rest of this paper is to prove that the representations $\style{U}(\delta,2)$ do indeed appear in $L^2_{\operatorname{disc}}(X)$.

\begin{rmk}
We do not show that generalized Speh representations	 $\style{U}(\delta,2m)$, $m\geq 2$, are not relatively discrete despite the fact that \Cref{SV-discrete} predicts that these representations do not appear in $L^2_{\operatorname{disc}}(X)$.  See \cite[Remark 6.6]{smith2018} for a discussion of the difficulties therein.
\end{rmk}

\section{Tori and parabolic subgroups: structure of $\Sp_{2n}(F)\backslash \GL_{2n}(F)$}\label{sec-tori-pblc}
In this section, we identify the $\theta$-split parabolic subgroups required for our application of the Relative Casselman Criterion.
First we introduce a second involution that is $G$-equivalent to $\theta$ (\textit{cf}.~\Cref{sec-conventions}).
Let $w_+ \in G$ be the permutation matrix associated to the permutation
\begin{align*}
	\begin{cases} 2i-1 \mapsto i &  1\leq i \leq n \\
				  2i \mapsto 2n+1-i &  1 \leq i \leq n	
	\end{cases}
\end{align*}
of $\{1, \ldots, 2n\}$.  We have chosen $w_+$ such that
\begin{align*}
\ve_{2n} \cdot w_+ & = \tran{w_+} \ve_{2n} w_+ 
 = w_+^{-1} \ve_{2n} w_+ 
 = \left ( \begin{matrix} 0 & 1 & & &	\\
 						   -1 & 0 & & & \\
 						   & & \ddots & & \\
 						   	& & & 0 & 1 \\
 						  & & & -1 & 0
 \end{matrix} \right).
\end{align*}
Let $x_{2n}$ denote the nonsingular skew-symmetric matrix $\ve_{2n} \cdot w_+$ and let $\theta_{x_{2n}}$ be the associated involution of $G$.  
As above, we have that $\theta_{x_{2n}} = \theta_{\ve_{2n}\cdot w_+}= w_+ \cdot \theta$, and $\theta_{x_{2n}}$ is $G$-equivalent to $\theta$.

\begin{lem}\label{lem-split-torus}
Let $A_0$ be the maximal diagonal $F$-split torus of $G$.
The torus $A_0$ is $\theta$-stable and contains the maximal $(\theta,F)$-split torus $S_0$, where
\begin{align*}
S_0 = \{ \diag(a_1,\ldots, a_n, a_n,\ldots, a_1) : a_i \in F^\times, 1 \leq i \leq n \}.
\end{align*}
\end{lem}

\begin{proof}
Let $a = \diag(a_1,\ldots, a_{2n}) \in A_0$.
First note that 	
\begin{align*}
\theta(a) & = \diag(a_{2n}^{-1}, \ldots, a_1^{-1}).
\end{align*}
In particular, $a$ is $\theta$-split if and only if $a_{2n+1-i} = a_i$ for all $1 \leq i \leq n$.
The torus $S_0$ is the $(\theta,F)$-split component of $A_0$.  Thus, it is sufficient to show that $S_0$ is a maximal $(\theta,F)$-split torus in $G$.
To do so, we will prove that the block-upper triangular parabolic $P_{(\underline{2})}$ corresponding to the partition $(\underline{2}) = (2,\ldots,2)$ of $2n$ is a minimal $\theta_{x_{2n}}$-split parabolic, and then use the $G$-equivalence of $\theta_{x_{2n}}$ and $\theta$ to conclude that $P_0 = w_+ P_{(\underline{2})} w_+^{-1}$ is a minimal $\theta$-split parabolic subgroup of $G$.
The desired result then follows from \cite[Proposition 4.7(iv)]{helminck--wang1993}.

To see that $P_{(\underline{2})}$ is $\theta_{x_{2n}}$-split, first note that $x_{2n} \in M_{(\underline{2})}$; therefore, the block-diagonal Levi $M_{(\underline{2})}$ is $\theta_{x_{2n}}$-stable.  The unipotent radical $N_{(\underline{2})}$ of $P_{(\underline{2})}$ is mapped to the opposite unipotent radical $N_{(\underline{2})}^{\operatorname{op}}$ (with respect to $M_{(\underline{2})}$) by taking conjugate-transpose, and both $N_{(\underline{2})}$ and $N_{(\underline{2})}^{\operatorname{op}}$ are normalized by $M_{(\underline{2})}$.  It follows that $\theta_{x_{2n}}(P_{(\underline{2})}) = M_{(\underline{2})}N_{(\underline{2})}^{\operatorname{op}} = P_{(\underline{2})}^{\operatorname{op}}$ and $P_{(\underline{2})}$  is $\theta_{x_{2n}}$-split.  
It only remains to show that $P_{(\underline{2})}$ is a minimal $\theta_{x_{2n}}$-split parabolic subgroup.
Suppose that $P=MN \subsetneq P_{(\underline{2})}$ is a $\theta_{x_{2n}}$-split parabolic subgroup of $G$ that is properly contained in $P_{(\underline{2})}$.  The parabolic subgroup $P \cap M_{(\underline{2})}$ of  $M_{(\underline{2})}$ is $\theta_{x_{2n}}$-split in $M_{(\underline{2})}$.  Notice that the $\mathrm{GL}$-blocks of $M_{(\underline{2})}$ are not interchanged by $\theta_{x_{2n}}$. 
In fact, $\theta_{x_{2n}}$ restricted to $M_{(\underline{2})}$ is equal to  the product $\theta_{x_2} \times \ldots \times \theta_{x_2}$. It follows that  $P \cap M_{(\underline{2})}$ is a product of $\theta{x_2}$-split parabolic subgroups in $\GL_2(F)$. Notice that the $F$-split component of the centre of $M_{(\underline{2})}$ is $(\theta_{x_{2n}},F)$-split. By \cite[Proposition 4.7(iv)]{helminck--wang1993}, no proper parabolic subgroup of $\GL_2(F)$ can be $\theta{x_2}$-split, and it follows that $M_{(\underline{2})}$ has no proper $\theta_{x_{2n}}$-split parabolic subgroups. In particular, $P_{(\underline{2})}$ is a  minimal $\theta_{x_{2n}}$-split parabolic subgroup of $G$.
\end{proof}

The torus $S_{0,x_{2n}} = \{\diag(a_1,a_1,\ldots, a_n,a_n) : a_i \in F^\times\}$ is a maximal $(\theta_{x_{2n}},F)$-split torus of $G$, it is the $(\theta_{x_{2n}},F)$-split component of $P_{(\underline{2})}$ and the $F$-split component of $M_{(\underline{2})}$.
The torus $S_0$ is the $w_+$-conjugate of $S_{0,x_{2n}}$.  We also note explicitly that $P_0 = w_+P_{(\underline{2})}w_+^{-1}$ is $\theta$-split:
\begin{align*}
\theta(P_0) &= \theta (w_+P_{(\underline{2})}w_+^{-1}) \\
& = w_+ w_+^{-1} \theta (w_+P_{(\underline{2})}w_+^{-1}) w_+ w_+^{-1} \\
& = w_+ \theta_{x_{2n}}(P_{(\underline{2})}) w_+^{-1} \\
& =  w_+ (P_{(\underline{2})}^{\operatorname{op}}) w_+^{-1} \\
& = P_0^{\operatorname{op}},
\end{align*}
where the opposite is taken with respect to the $\theta$-stable Levi factor $M_0 = w_+ M_{(\underline{2})} w_+^{-1}$.
Let $N_0 = w_+ N_{(\underline{2})} w_+^{-1}$ denote the unipotent radical of $P_0$.
We emphasize that $P_0=M_0N_0$ is a minimal $\theta$-split parabolic subgroup of $G$.
\subsection{The restricted root system and $\theta$-split parabolic subgroups}
\begin{defn}
Let $\Delta$ be a base of a root system $\Phi$.
The $\Delta$-positive (respectively, $\Delta$-negative) roots in $\Phi$ consist of the collection of positive (respectively, negative) roots in $\Phi$ with respect to $\Delta$; in particular, the set of $\Delta$-positive roots is equal to $\Phi \cap \spn_{\Z_{\geq 0}} \Delta$.
\end{defn}

Let $\Phi_0 = \Phi(\G, \Abf_0)$ be the root system of $G$ with respect to $A_0$.  
Since $A_0$ is $\theta$-stable, the involution $\theta$ acts on $X^*(A_0)$ and $\Phi_0$ is $\theta$-stable under this action.
Let $\Delta = \{ \epsilon_i - \epsilon_{i+1} : 1 \leq i \leq 2n-1\}$ be the standard base for $\Phi_0$, where $\epsilon_i$ denotes the $i$-{th} $F$-rational coordinate character of $A_0$.
Define $\Delta_0 = w_+\Delta$ to be the Weyl group translate of $\Delta$ by the permutation matrix $w_+ \in W_0$, where $W_0 \cong N_G(A_0) / A_0$ is the Weyl group of $G$ (with respect to $A_0$).  We identify $W_0$ with the subgroup of $G$ consisting of all permutation matrices.

\begin{lem}
The set $\Phi_0^\theta$ of $\theta$-fixed roots in $\Phi_0$ is equal to the set
\begin{align*}
\Phi_0^\theta & = \{ \epsilon_i - \epsilon_{2n+1-i} : 1 \leq i \leq 2n \},	
\end{align*}
corresponding to the root spaces on the main anti-diagonal in $\gl_{2n}$.
\end{lem}

\begin{proof}
For any $1 \leq i \neq j \leq 2n$, we have that
$\theta(\epsilon_i - \epsilon_j)  = \epsilon_{2n+1-j} - \epsilon_{2n+1-i}$.
Note that $2n+1 - (2n+1 - i) = i$; therefore,
the root $\epsilon_i - \epsilon_j$ is $\theta$-fixed if and only if $j=2n+1-i$.
\end{proof}

\begin{lem}\label{lem-theta-base}
The set of simple roots $\Delta_0= w_+\Delta$ is a $\theta$-base of $\Phi_0$.
\end{lem}

\begin{proof}
The set $\Phi_0^+$ of $\Delta_0$-positive roots is equal to $w_+ \Phi_{\Delta}^+$, where $\Phi_{\Delta}^+$ is the set of $\Delta$-positive roots.
Moreover, the set of $\Delta_0$-negative roots in $\Phi_0$ is $\Phi_0^- = - \Phi_0^+ =  w_+ \Phi_\Delta^-$.
Let $\alpha = \epsilon_i - \epsilon_j \in \Phi_\Delta^+$, that is, $1 \leq i < j \leq 2n$ and $w_+\alpha \in \Phi_0^+$. 
Suppose that $w_+\alpha$ is not $\theta$-fixed. 
Note that $w_+ \epsilon_i = \epsilon_{w_+(i)}$ and thus
\begin{align*}
\theta(w_+(\epsilon_i-\epsilon_j)) & = \epsilon_{2n+1-w_+(j)} - \epsilon_{2n+1-w_+(i)}	
\end{align*}
We consider the image of $w_+ \alpha$ under $\theta$ in the following four cases.
\begin{description}
	\item[Case (i): $i,j$ both odd] We can write $i=2k-1$ and $j=2l-1$ with $1\leq k < l \leq n$.  It follows that
	\begin{align*}
\theta(w_+\alpha) 
	& = \epsilon_{2n+1-l} - \epsilon_{2n+1-k}	
	 = w_+ ( \epsilon_{2l} - \epsilon_{2k}	);
\end{align*}
moreover, since $2l > 2k$, we have that $ w_+ ( \epsilon_{2l} - \epsilon_{2k}	) \in \Phi_0^-$.
\item[Case (ii): $i$ odd, $j$ even] 
	Let $i=2k-1$ and $j=2l$ with $1\leq k \leq l \leq n$.  As above,
	\begin{align*}
\theta(w_+\alpha) 
	& = \epsilon_{l} - \epsilon_{2n+1-k}	 
	 = w_+ ( \epsilon_{2l-1} - \epsilon_{2k}	);
\end{align*}
Observe that $k \neq l$, since otherwise $w_+\alpha = \theta(w_+\alpha) \in \Phi_0^\theta$ and we have assumed that $w_+\alpha$ is not $\theta$-fixed.
Since $l > k$, we have $2l-1>2k$ and $ w_+ ( \epsilon_{2l} - \epsilon_{2k}	) \in \Phi_0^-$.
	\item[Case (iii): $i$ even, $j$ odd] Let $i=2k$ and $j=2l-1$ where $1\leq k < l \leq n$.  It follows that
	\begin{align*}
\theta(w_+\alpha) 
	& = \epsilon_{2n+1-l} - \epsilon_{2n+1-(2n+1-k)}	 
	 = w_+ ( \epsilon_{2l} - \epsilon_{2k-1}	);
\end{align*}
moreover, since $l > k$, we have $2l>2k-1$ and $w_+ ( \epsilon_{2l} - \epsilon_{2k-1}	) \in \Phi_0^-$.
	\item[Case (iv): $i,j$ both even]	Let $i=2k$ and $j=2l$ for $1\leq k < l \leq n$.  We have
	\begin{align*}
\theta(w_+(\epsilon_{2k}-\epsilon_{2l})) 
	& = \epsilon_{2n+1-(2n+1-l)} - \epsilon_{2n+1-(2n+1-k)}	
	 = w_+ ( \epsilon_{2l-l} - \epsilon_{2k-1}	);
\end{align*}
moreover, since $l > k$, we have that $2l-1>2k-1$ and $w_+ ( \epsilon_{2l-1} - \epsilon_{2k-1}	) \in \Phi_0^-$.
\end{description}
It follows that if $\beta \in \Phi_0^+$ is not $\theta$-fixed, then $\theta(\beta) \in \Phi_0^-$; therefore, $\Delta_0$ is a $\theta$-base of $\Phi_0$.
\end{proof}

\begin{obs}
From the proof of \Cref{lem-theta-base}, we see that the set of $\theta$-fixed $\Delta_0$-positive roots are the translates of $\{\epsilon_1 - \epsilon_2, \epsilon_3 - \epsilon_4, \ldots, \epsilon_{2n-1} - \epsilon_{2n}\}$ by $w_+$.	The subset $\{\epsilon_1 - \epsilon_2, \epsilon_3 - \epsilon_4, \ldots, \epsilon_{2n-1} - \epsilon_{2n}\}$ of $\Delta$ consists of $\theta_{x_{2n}}$-fixed roots and determines the (minimal $\theta_{x_{2n}}$-split) parabolic subgroup $P_{(\underline{2})}$.
\end{obs}

To aid in our understanding of the structure of $\Delta_0$, we partition the roots in the standard base $\Delta$ into the disjoint subsets
\begin{align*}
\Delta_{\operatorname{odd}} & = \{\epsilon_{2i-1} - \epsilon_{2i} : 1 \leq i \leq n\}
\end{align*}
and
\begin{align*}
\Delta_{\operatorname{even}} & = \{ \epsilon_{2j}-\epsilon_{2j+1} : 1\leq j \leq n-1 \}.	
\end{align*}
 Notice that the set of $\theta$-fixed simple roots in $\Delta_0$ is equal to $\Delta_0^\theta = w_+ \Delta_{\operatorname{odd}}$.
Moreover, $\Delta_0$ is the disjoint union $\Delta_0 = \Delta_0^\theta \sqcup w_+  \Delta_{\operatorname{even}}$.
Explicitly,
\begin{align*}
\Delta_0^\theta & = w_+  \Delta_{\operatorname{odd}} = \{ \epsilon_i - \epsilon_{2n+1-i} : 1 \leq i \leq n\} 
\end{align*}
and
\begin{align*}
w_+ \Delta_{\operatorname{even}} & = \{ \epsilon_{2n+1-j} - \epsilon_{j+1} : 1 \leq j \leq n-1 \}.	
\end{align*}

Let $r: X^*(A_0) \rightarrow X^*(S_0)$ be the surjective homomorphism defined by restricting $F$-rational characters of $A_0$ to $S_0$. The $\theta$-fixed simple roots are trivial on $S_0$.  It follows that
\begin{align*}
\overline\Delta_0 & = r(\Delta_0 \setminus \Delta_0^\theta) = r(w_+ \Delta_{\operatorname{even}}) = \{ \bar\epsilon_i - \bar\epsilon_{i+1} : 1 \leq i \leq n-1 \},
\end{align*}
where $\bar\epsilon_i$ is the $i$-{th} $F$-rational coordinate character of $S_0$ given by
\begin{align*}
\bar\epsilon_i(\diag(a_1,\ldots,a_n,a_n, \ldots, a_1)) & = a_i,	
\end{align*}
for $1 \leq i \leq n$. In addition, the full set of restricted roots is
\begin{align*}
\overline \Phi_0 &= r(\Phi_0) \setminus \{0\} 
 = r(\Phi_0 \setminus \Phi_0^\theta)
 = \{ \bar \epsilon_{i} - \bar \epsilon_{j}  : 1 \leq i \neq j \leq n \}.	
\end{align*}

We have established the following.

\begin{lem}\label{lem-rest-roots-dual-gp}
The restricted root system associated to $\Sp_{2n}(F) \backslash \GL_{2n}(F)$ is of type ${\rm{A}}_{n-1}$ and the dual group $G_{X}^\vee$ of $X=\Sp_{2n}(F) \backslash \GL_{2n}(F)$ is $\mathrm{GL}(n,\C)$.	
\end{lem}

Proper ($\Delta_0$-)standard $\theta$-split parabolic subgroups of $G$ are parametrized by proper $\theta$-split subsets of $\Delta_0$, where a subset $\Theta$ of $\Delta_0$ is $\theta$-split if it is of the form
\begin{align*}
\Theta = [\overline\Theta]:=r^{-1}(\overline\Theta)\cup \Delta_0^\theta,	
\end{align*}
and  $\overline\Theta$ is  a subset of $\overline \Delta_0$.
The subset $\Delta_0^\theta$ of $\theta$-fixed simple roots determines the minimal standard $\theta$-split parabolic $P_0 = M_0 N_0$ of $G$, with Levi factor $M_0 = C_G(S_0)$ and unipotent radical $N_0$.
By \cite[Lemma 2.5]{kato--takano2008}, any $\theta$-split parabolic subgroup of $G$ is $(\Hbf \Mbf_0)(F)$-conjugate to a standard $\theta$-split parabolic.
In the current setting, the Galois cohomology of $\Mbf_0 \cap \Hbf$ over $F$ is trivial and it follows that $(\Hbf \Mbf_0)(F) = HM_0$; moreover, any $\theta$-split parabolic subgroup is $H$-conjugate to a standard $\theta$-split parabolic subgroup.
For completeness, we give a proof.

\begin{lem}\label{lem-trivial-cohom}
The first Galois cohomology of $\Mbf_0 \cap \Hbf$ over $F$ is trivial and $(\Hbf \Mbf_0)(F) = HM_0$.
\end{lem}

\begin{proof}
First, one may readily verify that
\begin{align*}
\Mbf_0 \cap \Hbf & = w_+ \left(\Mbf_{(\underline{2})} \cap \G^{\theta_{x_{2n}}}\right) w_+^{-1} \cong \prod_1^n \mathbf{SL}_2.
\end{align*}
By Hilbert's Theorem 90, it follows that
\begin{align*}
H^1(\Mbf_0 \cap \Hbf, F) & \cong \oplus_1^n H^1(\mathbf{SL}_2,F) = 0.	
\end{align*}
Let $\bar F$ denote the algebraic closure of $F$.
By considering the long exact sequence in Galois cohomology obtained from the short exact sequence
\begin{align*}
1 \rightarrow \Mbf_0(\bar F) \cap \Hbf(\bar F) \rightarrow  \Hbf(\bar F) \times \Mbf_0(\bar F) \rightarrow \Hbf(\bar F) \Mbf_0(\bar F) \rightarrow 1	
\end{align*}
of pointed sets, it follows that $(\Hbf \Mbf_0)(F) = HM_0$, as claimed.
\end{proof}

\begin{prop}\label{cor-H-cong-pblc}
Let $P$ be a $\theta$-split parabolic subgroup of $G$.
There exists a $\theta$-split subset $\Theta$ of $\Delta_0$ and an element $h\in H$ such that $P = h P_\Theta h^{-1}$.
Moreover, $P$ has unipotent radical $N = hN_\Theta h^{-1}$ and $\theta$-stable Levi factor $M = hM_\Theta h^{-1}$.
\end{prop}

\begin{proof}
Apply \Cref{lem-trivial-cohom}  and \cite[Lemma 2.5]{kato--takano2008}.	
\end{proof}

With the last result in hand, we explicitly determine the maximal proper standard $\theta$-split parabolic subgroups of $G$ which correspond to the maximal proper $\theta$-split subsets of $\Delta_0$.
A maximal proper $\theta$-split subset of $\Delta_0$ has the form $[\overline\Delta_0\setminus \{\bar\alpha\}] = r^{-1}(\overline\Delta_0\setminus \{\bar\alpha\}) \cup \Delta_0^\theta$, where $\bar\alpha \in \overline\Delta_0$.
Observe that for each $\bar\alpha \in \overline{\Delta}_0$ there is a unique $\alpha \in w_+\Delta_{\operatorname{even}}$ such that $r(\alpha) = \bar \alpha$.
Precisely, the pre-image of $\bar\epsilon_i-\bar\epsilon_{i+1}$ under the restriction map $r:X^*(A_0)\rightarrow S^*(S_0)$ is $r^{-1}(\bar\epsilon_i-\bar\epsilon_{i+1}) = w_+(\epsilon_{2i}-\epsilon_{2i+1})$, for each $1\leq i \leq n-1$.
It follows that for each $1\leq k \leq n-1$ we have a maximal $\theta$-split subset of $\Delta_0$ given by
\begin{align}\label{eq-max-theta-split-subset}
\Theta_k & = r^{-1} (\overline \Delta_0 \setminus \{ \bar\epsilon_k - \bar\epsilon_{k+1}\}) \cup \Delta_0^{\theta} \\
\nonumber & = w_+( \Delta \setminus \{\epsilon_{2k}-\epsilon_{2k+1}\})\\
\nonumber & = \Delta_0 \setminus \{ \epsilon_{2n+1-k}-\epsilon_{k+1} \}.
\end{align}
To each $\Theta_k$, $1\leq k \leq n-1$, we associate the maximal $\Delta_0$-standard $\theta$-split parabolic subgroup 
\begin{align*}
	P_{\Theta_k} &= w_+ P_{(2k,2n-2k)} w_+^{-1},
\end{align*}
with $\theta$-stable Levi factor $M_{\Theta_k} = w_+ M_{(2k,2n-2k)}w_+^{-1}$ and unipotent radical $N_{\Theta_k} = w_+ N_{(2k,2n-2k)} w_+^{-1}$.
Notice that $P_{\Theta_k}$ does indeed contain the minimal standard $\theta$-split parabolic subgroup $P_0 = w_+ P_{(\underline{2})} w_+^{-1}$, corresponding to $\Delta_0^\theta$ (or the partition $(\underline{2}) = (2,\ldots,2)$ of $2n$).
Moreover, by \Cref{lem-split-torus}, the $(\theta,F)$-split component $S_{\Theta_k}$ of $P_{\Theta_k}$ is equal to its $F$-split component $A_{\Theta_k}$.
\begin{note}
It may be helpful to observe that the maximal $\theta_{x_{2n}}$-split subsets of $\Delta$ are thus given by $\Delta \setminus \{\epsilon_{2k}-\epsilon_{2k+1}\}$, where $1\leq k \leq n-1$. 	It follows that the standard block-upper-triangular parabolic subgroups $P_{(2k,2n-2k)}$, with even sized blocks, are the maximal $\Delta$-standard $\theta_{x_{2n}}$-split parabolic subgroups.
\end{note}
\subsection{Inducing from distinguished representations of $\theta$-elliptic Levi subgroups}\label{sec-elliptic-context}
We recall the following definition.
\begin{defn}
A $\theta$-stable Levi subgroup of $G$ is $\theta$-elliptic if $L$ is not contained in any proper $\theta$-split parabolic subgroup of $G$.	
\end{defn}

In order to place the Speh representations within the context of the relative discrete series constructed in \cite{smith2018,smith2018b}, we show that $\style{U}(\delta,2)$ can be realized as the quotient of a representation induced from a distinguished representation of a $\theta$-elliptic Levi subgroup.

\begin{lem}
The block-upper triangular parabolic subgroup $P_{(n,n)}$, corresponding to $\Omega^{\operatorname{ell}} = \Delta \setminus \{\epsilon_n - \epsilon_{n-1}\} \subset \Delta$,  is $\theta$-stable and the $\Delta$-standard block-diagonal Levi subgroup $M_{(n,n)}$ is $\theta$-elliptic.	
\end{lem}

\begin{proof}
First, it is clear that $P_{(n,n)}$ and $M_{(n,n)}$ are $\theta$-stable subgroups of $G$.
It is readily verified that the $(\theta,F)$-split component of $M_{(n,n)}$ is equal to the $(\theta,F)$-split component $S_G$ of $G$; moreover, $S_G = A_G$, that is, the $(\theta,F)$-split component of $G$ is equal to the $F$-split component of $G$.
By \cite[Lemma 3.8]{smith2018}, the $\theta$-stable Levi subgroup $M_{(n,n)}$ is $\theta$-elliptic.
\end{proof}

In what follows, we let $Q = P_{(n,n)} = P_{\Omega^{\operatorname{ell}}}$, $L = M_{(n,n)}=M_{\Omega^{\operatorname{ell}}}$, and $U=N_{(n,n)}=N_{\Omega^{\operatorname{ell}}}$.
Define $\Omega = w_+ \Omega^{\operatorname{ell}} \subset \Delta_0$.  We then have that $Q = w_+^{-1} P_\Omega w_+$ or, equivalently, that $P_\Omega = w_+ Q w_+^{-1}$.

\begin{defn}
	An ordered partition $(m_1, \ldots, m_k)$ of an integer $m\geq 2$ is balanced if $(m_1,\ldots, m_k)$ is equal to the opposite partition $(m_1,\ldots, m_k)^{\operatorname{op}} = (m_k,\ldots, m_1)$.
\end{defn}

\begin{lem}
Let $P$ be a block-upper triangular ($\Delta$-standard) parabolic subgroup of $G$.  The subgroup $P$ is $\theta$-stable if and only if $P$ corresponds to a balanced partition of $2n$.  In addition, the only $\theta$-stable $\Delta$-standard maximal parabolic that admits a $\theta$-elliptic Levi subgroup is $P_{(n,n)}$.
\end{lem}

\begin{proof}
The proof is the same as that of \cite[Lemma 4.15]{smith2018b}. 
\end{proof}

Recall that a parabolic subgroup $P$ is $A_0$-semi-standard if $P$ contains the maximal $F$-split torus $A_0$.
In particular, the $\Delta$- and $\Delta_0$-standard parabolic subgroups are $A_0$-semi-standard.
The next result is the analogue of \cite[Lemma 4.21]{smith2018b}; the proof is the same.

\begin{lem}
Let $P$ be any $\theta$-stable parabolic subgroup of $G$.  	The subgroup $P$ is $H$-conjugate to a $\theta$-stable $A_0$-semi-standard parabolic subgroup.
\end{lem}

\begin{lem}
The $\theta$-stable Levi subgroup $L=M_{(n,n)}$ is the only proper $\theta$-elliptic $A_0$-semi-standard Levi subgroup of $G$ up to conjugacy by Weyl group elements $w \in W_0 = W(G,A_0) = N_G(A_0)/A_0$ such that $w^{-1}\ve_{2n}w \in N_G(L) \setminus L$.
\end{lem}

\begin{proof}
See the proof of \cite[Lemma 4.20(2)]{smith2018b}.	
\end{proof}

\begin{lem}\label{lem-L-fixed}
The group $L^\theta$ of $\theta$-fixed points in $L=M_{(n,n)}$ is isomorphic to $\GL_n(F)$ embedded in $L$ as follows:
\begin{align*}
L^\theta  & = \left\{ \left( \begin{matrix} g & 0 \\ 0 & J_n^{-1}\tran{g}^{-1} J_n \end{matrix} \right) : g \in \GL_n(F) \right\}.	
\end{align*}
\end{lem}

\begin{proof}
We omit the straightforward calculation.
\end{proof}

\begin{prop}\label{prop-L-dist}
Let $\tau_1 \otimes \tau_2$ be an irreducible admissible representation of $L= M_{(n,n)}$.  
Then $\tau_1\otimes \tau_2$ is $L^\theta$-distinguished if and only if $\tau_2 \cong \tau_1$.	
\end{prop}

\begin{proof}
First, one can show that $\tau_1\otimes \tau_2$ is $L^\theta$-distinguished if and only if $\tau_2 \cong \wt\tau_1\circ \theta_{J_n}$, where $\theta_{J_n}$ is the involution on $\GL_n(F)$ given by $\theta_{J_n}(g) = J_n^{-1} \tran{g}^{-1} J_n$, for $g\in \GL_n(F)$.
Now, the lemma is a simple consequence of \cite[Theorem 2]{gelfand--kazhdan1975} which implies that $\wt\tau_1 \cong \tau_1 \circ \tran(\cdot)^{-1}$ and the fact that $J_n^{-1} = J_n = \tran{J_n}$ (see \cite[Lemma 5.3]{smith2018b} for addition detail).
\end{proof}

Let $\tau$ be an irreducible admissible representation of $\GL_n(F)$.  The representation $\tau\otimes \tau$ of $L$ is $L^\theta$-distinguished by \Cref{prop-L-dist}.  Moreover, the $L^\theta$-invariant linear form on $\tau\otimes\tau$ can be realized via the standard pairing between $\tau$ and its contragredient $\wt\tau$.  Indeed, this follows from  \cite[Theorem 2]{gelfand--kazhdan1975} and the fact that $\wt\tau \cong \tau \circ \theta_{J_n}$.
Let $\lambda_\tau \in \Hom_{L^\theta}(\tau\otimes\tau, 1)$ be the (nonzero) invariant form that arises via the pairing on $\tau\otimes\wt\tau$.  
Let $l=\diag(x, \theta_{J_n}(x)) \in L^\theta$ and consider the value of $\delta_{Q^\theta}\delta_Q^{-1/2}$ on $l$.
It is straightforward to check that 
\begin{align*}
\left(\delta_{Q^\theta} \left.\delta_Q^{-1/2} \right\vert_{L^\theta}\right)	(l) = |\det(x)|^{n+1}|\det(x)|^{-n} = |\det(x)| = \nu(x),
\end{align*}
that is, $\delta_{Q^\theta} \delta_Q^{-1/2}$ agrees with the character $\nu$ on $\GL_n(F) \cong L^\theta$.
Since the contragredient of $\nu$ is $\nu^{-1}$, it follows that $\lambda_\tau \in \Hom_{L^\theta}(\nu^{1/2} \tau \otimes \nu^{-1/2}\tau, \nu) \cong \Hom_{L^\theta}(\delta_Q^{1/2}\tau\otimes\tau,\delta_{Q^\theta})$.
By \cite[Proposition 7.1]{offen2017}, $\lambda_\tau$ maps to a nonzero $H$-invariant linear form $\lambda \in \Hom_H(\nu^{1/2} \tau \times \nu^{-1/2}\tau ,1)$, and   the parabolically induced representation $\nu^{1/2}\tau \times \nu^{-1/2}\tau = \iota_{Q}^{G}\left(\nu^{1/2}\tau\otimes\nu^{-1/2}\tau\right)$ is $H$-distinguished.
We now state a result of Heumos and Rallis \cite[Theorem 11.1]{heumos--rallis1990} (\textit{cf.}~\Cref{sec-speh}).
We give a sketch of the proof (still appealing to the main results of \cite{heumos--rallis1990}).

\begin{prop}(Heumos--Rallis)\label{prop-speh-dist}
Let $\delta$ be an irreducible square integrable representations of $\GL_n(F)$.  The parabolically induced representation $\nu^{1/2}\delta \times \nu^{-1/2}\delta = \iota_{Q}^{G}\left(\nu^{1/2}\delta\otimes\nu^{-1/2}\delta\right)$ is $H$-distinguished.  Moreover, the unique irreducible quotient $\style{U}(\delta,2)$ of  $\nu^{1/2}\delta \times \nu^{-1/2}\delta$ is $H$-distinguished.
\end{prop}

\begin{proof}
As above, $\nu^{1/2}\delta \times \nu^{-1/2}\delta$ is $H$-distinguished by \Cref{prop-L-dist} and \cite[Proposition 7.1]{offen2017}.  
The parabolically induced representation $\nu^{1/2}\delta \times \nu^{-1/2}\delta$ has length two \cite{bernstein--zelevinsky1977,zelevinsky1980}.  Let $\style{Z}(\delta,2)$ be the unique irreducible subrepresentation  and let $\style{U}(\delta,2)$  be the unique irreducible quotient of  $\nu^{1/2}\delta \times \nu^{-1/2}\delta$.
The subrepresentation $\style{Z}(\delta,2)$ is tempered and thus generic \cite[Theorem 9.3]{zelevinsky1980}.  
Therefore, by \cite[Theorem 3.2.2]{heumos--rallis1990},  $\style{Z}(\delta,2)$ cannot be $H$-distinguished.
 It follows that any nonzero $H$-invariant linear form on $\nu^{1/2}\delta \times \nu^{-1/2}\delta$ descends to a well-defined nonzero $H$-invariant linear functional on the quotient $\style{U}(\delta,2)$.
\end{proof}

\begin{rmk}
By the multiplicity-one result \cite[Theorem 2.4.2]{heumos--rallis1990}, the $H$-invariant linear form on $\nu^{1/2} \tau \times \nu^{-1/2}\tau$ constructed via \cite[Proposition 7.1]{offen2017} is a scalar multiple of the invariant form produced by Heumos and Rallis in \cite[$\S$11.3.1.2]{heumos--rallis1990} (\textit{cf.}~\cite[Lemma 1.3.4]{smith2020a}).
\end{rmk}

\section{Application of the Relative Casselman Criterion}\label{sec-rel-cass}
We now come to the main result of the paper.
\begin{thm}\label{thm-speh-rds}
Let $\delta$ be a discrete series representation of $\GL_n(F)$.  The Speh representation $\style{U}(\delta,2)$ of $\GL_{2n}(F)$ is $\Sp_{2n}(F)$-relatively discrete.
\end{thm}

\begin{proof}
Let $\lambda \in \Hom_H(\style{U}(\delta,2), 1)$ be nonzero.
Let $\pi = \nu^{1/2}\delta \times \nu^{-1/2}\delta$.  Recall from \Cref{sec-speh} that $\style{U}(\delta,2)$ is the unique irreducible quotient of $\pi$.
By \Cref{cor-H-cong-pblc} and \cite[Proposition 4.22]{smith2018}, it is enough to consider exponents along maximal standard $\theta$-split parabolic subgroups of $G$ when applying \Cref{thm-relative-casselman}  (\cite[Theorem 4.7]{kato--takano2010}).
Let $P=MN$ be a maximal $\Delta_0$-standard $\theta$-split parabolic subgroup of $G$ with unipotent radical $N$ and $\theta$-stable Levi factor $M=P\cap\theta(P)$.
By \cite[Proposition 4.23]{smith2018}, only exponents corresponding to irreducible $M^\theta$-distinguished subquotients of the Jacquet module $\style{U}(\delta,2)_{N}$ may appear in $\Exp_{S_M}(\style{U}(\delta,2)_{N}, \lambda_{N})$.
By \Cref{prop-no-dist}, the irreducible unitary subquotients of $\pi_N$, and also $\style{U}(\delta,2)$, are not $M^\theta$-distinguished.
By \Cref{prop-good-exp}, all exponents that appear in $\Exp_{S_M}(\style{U}(\delta,2)_{N}, \lambda_{N})$ satisfy \eqref{eq-rel-cass-condition}.
By \Cref{thm-relative-casselman}, $\style{U}(\delta,2)$ is $(H,\lambda)$-relatively square integrable.
Multiplicity-one holds by \cite[Theorem 2.4.2]{heumos--rallis1990}, thus $\dim \Hom_H(\style{U}(\delta,2),1) = 1$ and $\style{U}(\delta,2)$ is $H$-relatively square integrable.
\end{proof}

The remainder of the paper is dedicated to proving \Cref{prop-no-dist} and  \Cref{prop-good-exp}.

Let $\delta$ be an irreducible admissible square integrable (discrete series) representation of $\GL_n(F)$.
Let $\pi = \nu^{1/2}\delta \times \nu^{-1/2}\delta$.
The sequence
\begin{align*}
0 \rightarrow \style{Z}(\delta,2) \rightarrow 	\pi \rightarrow \style{U}(\delta,2) \rightarrow 0,
\end{align*}
of $G$-modules is exact, where $\style{Z}(\delta,2)$ is the unique irreducible generic subrepresentation of $\pi$ (see \Cref{sec-speh}).
We keep the notation of \Cref{sec-tori-pblc} and let $Q=P_{(n,n)}$, $L=M_{(n,n)}$, and $U = N_{(n,n)}$.
Let $P = MN$ be a maximal $\Delta_0$-standard $\theta$-split parabolic subgroup of $G$, with unipotent radical $N$ and  $\theta$-stable Levi factor $M = P\cap\theta(P)$. 
The Jacquet restriction functor (along $P$) is exact; therefore, we have an exact sequence of $M$-modules
\begin{align}\label{eq-jacquet-exact}
0 \rightarrow \style{Z}(\delta,2)_N \rightarrow 	\pi_N \rightarrow \style{U}(\delta,2)_N \rightarrow 0.
\end{align}
Our goal is to understand the irreducible subquotients, and the exponents, of $\style{U}(\delta,2)_N$ by applying the Geometric Lemma \cite[Lemma 2.12]{bernstein--zelevinsky1977} to $\pi_N$.
If $\chi \in \Exp_{A_M}(\style{U}(\delta,2)_N)$, then $\chi$ is the central quasi-character of an irreducible subquotient of $\style{U}(\delta,2)_N$ and thus of $\pi_N$, that is, $\chi$ appears in $\Exp_{A_M}(\pi_N)$.
Recall that we can realize $Q = w_+^{-1} P_\Omega w_+$, where $\Omega = \Delta_0 \setminus \{ w_+(\epsilon_n - \epsilon_{n+1}) \}$, and $P = P_\Theta$, for some $1 \leq k \leq n-1$, where $\Theta = \Theta_k$ is described in \eqref{eq-max-theta-split-subset}.
In particular, $\Omega$ and $\Theta$ are subsets of the $\theta$-base $\Delta_0$.
Let 
\begin{align*}
	[W_\Theta \backslash W_0 / W_\Omega] & = \{ w \in W_0 : w \Omega \subset \Phi_0^+, w^{-1}\Theta \subset \Phi_0^{-1} \},
\end{align*}
where $\Phi_0^+$ is the set of $\Delta_0$-positive roots.
By \cite[Propositions 1.3.1 and 1.3.3]{Casselman-book}, the set $[W_\Theta \backslash W_0 / W_\Omega]\cdot w_+$ is a system of representatives for $P \backslash G / Q$.
By the Geometric Lemma \cite[Lemma 2.12]{bernstein--zelevinsky1977}, there exists a filtration of the space of $\pi_N$ such that the associated graded object $\operatorname{gr}(\pi_N)$ is isomorphic to
\begin{align}\label{eq-full-ind-geom-lemma}
	\bigoplus_{y \in [W_\Theta \backslash W_0 / W_\Omega]\cdot w_+} \iota_{M\cap {}^yQ}^M \left( {}^y (\nu^{1/2}\delta\otimes\nu^{-1/2}\delta)_{N \cap {}^yL} \right).
\end{align}
Write $\mathcal F_N^y(\delta,2)$ to denote the representation $\iota_{M\cap {}^yQ}^M \left( {}^y (\nu^{1/2}\delta\otimes\nu^{-1/2}\delta)_{N \cap {}^yL} \right)$.
Thus 
\begin{align*}
	\operatorname{gr}(\pi_N)
	& \cong \bigoplus_{w \in [W_\Theta \backslash W_0 / W_\Omega]} \mathcal F_N^{ww_+}(\delta,2).
\end{align*}
The exponents of $\pi$ along $P$ are the central characters of the irreducible subquotients of $\pi_N$; moreover, the exponents of $\style{U}(\delta,2)$ along $P$ are a subset of the of the exponents of $\pi$ along $P$.
Recall that, by \Cref{lem-split-torus}, the $(\theta,F)$-split component $S_M$ of $M$ is equal to its $F$-split component $A_M$; precisely, 
\begin{align*}
A_M & = w_+\{\diag(\underbrace{a,\ldots,a}_{2k},\underbrace{b,\ldots,b}_{2n-2k}) : a,b \in F^\times \}w_+^{-1} \\
& = \{\diag(\underbrace{a,\ldots,a}_{k},\underbrace{b,\ldots,b}_{2n-2k},\underbrace{a,\ldots,a}_{k}) : a,b \in F^\times \}.
\end{align*}
By \cite[Proposition 1.3.3]{Casselman-book}, with our choice $[W_\Theta \backslash W_0 / W_\Omega]\cdot w_+$ of representatives for $P \backslash G / Q$, if $y = ww_+$ where $w \in [W_\Theta \backslash W_0 / W_\Omega]$, then $M \cap {}^y Q$ is a parabolic subgroup of $M$ with Levi factor $M \cap {}^y L$ and unipotent radical $M \cap {}^y U$.  Similarly, $P \cap {}^y L$ is a parabolic subgroup of $L$ with Levi subgroup $M \cap {}^yL$ and unipotent radical $N \cap {}^y L$.
Explicitly, since $P = P_\Theta$ and $Q = w_+^{-1} P_\Omega w_+$, we see that
\begin{align*}
M \cap {}^yL & = M_\Theta \cap w M_\Omega w^{-1}
 = M_\Theta \cap M_{w\Omega} 
 = M_{\Theta \cap w\Omega},	
\end{align*}
\begin{align*}
N \cap {}^yL & = N_\Theta \cap w M_\Omega w^{-1}
 = N_\Theta \cap M_{w\Omega},
\end{align*}
and
\begin{align*}
M \cap {}^yU & = M_\Theta \cap w N_\Omega w^{-1}
= M_\Theta \cap N_{w\Omega}.
\end{align*}

Let  $w\in [W_\Theta \backslash W_0 / W_\Omega]$.  To achieve our goal, there are two cases that we need to consider.  
\begin{description}
	\item[Case 1]  $P_\Theta \cap {}^wM_\Omega = {}^wM_\Omega$.
	\item[Case 2]   $P_\Theta \cap {}^w M_\Omega \subsetneq {}^wM_\Omega$ is a proper parabolic subgroup of ${}^wM_\Omega$
\end{description}
In Case 1, we show that the associated irreducible subquotients of $\pi_N$ are not $M^\theta$-distinguished.
In Case 2, we show that the corresponding exponents of $\pi_N$ satisfy the condition \eqref{eq-rel-cass-condition}.
The exact sequence \eqref{eq-jacquet-exact} allows us to conclude that the same holds for $\style{U}(\delta,2)_N$.

\subsection{Case 1: no distinction}\label{sec-case-1}
Assume that $w\in [W_\Theta \backslash W_0 / W_\Omega]$ is such that $P_\Theta \cap {}^wM_\Omega = {}^wM_\Omega$.
Then $N_\Theta \cap {}^wM_\Omega = \{e\}$ and $M_\Theta \cap {}^w M_\Omega = {}^w M_\Omega = M_{w\Omega}$.
In particular, ${}^w M_\Omega \subset M_\Theta$ and since ${}^w M_\Omega$ is maximal it follows that ${}^wM_\Omega = M_\Theta \cong \GL_n(F) \times \GL_n(F)$.
Moreover, $M_\Theta$ and $M_\Omega$ are associate standard Levi subgroups isomorphic to $\GL_n(F) \times \GL_n(F)$.  
It follows that $n$ must be even, $k = n/2$, and $\Theta = \Theta_{n/2} = w_+ (\Delta \setminus \{ \epsilon_n - \epsilon_{n+1} \} ) = \Omega$.
That is, $M_\Theta = M_\Omega$ and $w$ lies in $[W_\Omega \backslash W_0 / W_\Omega ] \cap W(\Omega,\Omega)$, where $W(\Theta,\Omega) = \{ w \in W_0 : w\Omega = \Theta\}$.
Set $y = ww_+$.
Then $M_\Omega \cap {}^y Q = M_{w\Omega} = M_\Omega$ and $P_\Omega \cap {}^y L = M_{w\Omega} = M_\Omega$.
In this setting, 
\begin{align*}
	\mathcal F^y_\Omega	(\delta,2) 
	& = \iota_{M_\Omega}^{M_\Omega}({}^y(\nu^{1/2}\delta\otimes\nu^{-1/2}\delta)_{\{e\}} )
	 = {}^y(\nu^{1/2}\delta\otimes\nu^{-1/2}\delta),
\end{align*}
since $N_\Omega \cap {}^w_\Omega = N_\Omega \cap M_\Omega = \{e\}$.

\begin{prop}\label{prop-no-dist}
	Let $w \in [W_\Omega \backslash W_0 / W_\Omega ] \cap W(\Omega,\Omega)$ and set $y = ww_+$. Let $\tau$ be an irreducible admissible generic representation of $\GL_n(F)$.  The representation ${}^y(\nu^{1/2}\tau\otimes\nu^{-1/2}\tau)$ of $M_\Omega$ is not $M_\Omega^\theta$-distinguished, that is, $\Hom_{M_\Omega^\theta}({}^y(\nu^{1/2}\tau\otimes\nu^{-1/2}\tau),1) = 0$.
\end{prop}

\begin{proof}
First, recall that $n$ is even, and observe that $M_\Omega^\theta \cong \Sp_n(F) \times \Sp_n(F)$.
Indeed, $M_\Omega = w_+M_{(n,n)}w_+^{-1}$ and $m = w_+\underline{m}w_+^{-1} \in M_\Omega$ is $\theta$-fixed if and only if $\underline{m} \in M_{(n,n)}$ is fixed by $w_+\cdot \theta = \theta_{x_{2n}}$.
Recall (see \Cref{sec-tori-pblc}) that 
\begin{align*}x_{2n} = \ve_{2n}\cdot w_+ = \left ( \begin{matrix} 0 & 1 & & &	\\
 						   -1 & 0 & & & \\
 						   & & \ddots & & \\
 						   	& & & 0 & 1 \\
 						  & & & -1 & 0
 \end{matrix} \right) \in M_{(n,n)}
 \end{align*}
 and $\theta_{x_{2n}}(g) = x_{2n}^{-1} \tran g^{-1} x_{2n}$.
 One may readily verify that the image of $\underline{m}=\diag(m_1,m_2) \in M_{(n,n)}$ under $\theta_{x_{2n}}$ is given by
 \begin{align*}
  \theta_{x_{2n}}(\underline{m}) = \diag(x_n^{-1} \tran m_1^{-1} x_n, 	x_n^{-1} \tran m_2^{-1} x_n) = \diag(\theta_{x_n}(m_1),\theta_{x_n}(m_2)).
 \end{align*}
 It follows that $\underline{m}$ is $\theta_{x_{2n}}$-fixed if and only if $m_i = \theta_{x_n}(m_i)$, for each $i=1,2$.
 Moreover, 
 \begin{align*}
 M_\Omega^\theta & = w_+ \left(M_{(n,n)}^{\theta_{x_{2n}}} \right) w_+^{-1}	 \\
 & = w_+ \left( \GL_n(F)^{\theta_{x_n}} \times \GL_n(F)^{\theta_{x_n}} \right) w_+^{-1} \\
 & \cong  \left( \Sp_n(F) \times \Sp_n(F) \right),
 \end{align*}
since $x_n \in \GL_n(F)$ is nonsingular and skew symmetric, and $ \GL_n^{\theta_{x_n}}  \cong \Sp_n$.

Next, we note that $[W_\Omega \backslash W_0 / W_\Omega ] \cap W(\Omega,\Omega)$ consists of two elements: the identity $e$ and $w_+w_{(n,n)}w_+^{-1}$, where
\begin{align*}
w_{(n,n)} & = \left(\begin{matrix} 0 & I_n \\ I_n & 0 \end{matrix}\right).	
\end{align*}
First, realize $[W_\Omega \backslash W_0 / W_\Omega] = w_+ [W_{\Omega^{\operatorname{ell}}} \backslash W_0 / W_{\Omega^{\operatorname{ell}}} ] w_+^{-1}$, and $W(\Omega, \Omega) = w_+W({\Omega^{\operatorname{ell}}},{\Omega^{\operatorname{ell}}})w_+^{-1}$, where ${\Omega^{\operatorname{ell}}}  = \Delta \setminus \{\epsilon_n-\epsilon_{n+1}\} = w_+^{-1}\Omega$.  If $w\in W_0$, then we identify $w$ with a permutation of $\{1,\ldots, 2n\}$ and note that $w(\epsilon_i) = \epsilon_{w(i)}$.
The set of $\Delta$-positive roots in $\Phi_0$ is $\Phi_\Delta^+ = \{ \epsilon_i - \epsilon_j : 1\leq i < j \leq 2n \}$.
Thus, by definition, $w \in W_0$ lies in the set $[W_{\Omega^{\operatorname{ell}}} \backslash W_0 / W_{\Omega^{\operatorname{ell}}}]$ if and only if $w(i) < w(i+1)$ and $w^{-1}(i) < w^{-1}(i+1)$, for all $1 \leq i \leq n-1$ and $n+1 \leq i \leq 2n-1$   (with $i\neq n$ since $\epsilon_n-\epsilon_{n+1} \notin \Omega^{\operatorname{ell}}$).
It is not difficult to verify that $[W_{\Omega^{\operatorname{ell}}} \backslash W_0 / W_{\Omega^{\operatorname{ell}}}]$ consists of the $n+1$ permutation matrices of the form
\begin{align}\label{eq-nice-permutations-case-1}
	\left(
	\begin{matrix} I_j & 0 & 0 & 0 \\
					0 & 0 & I_{n-j} & 0 \\
					0& I_{n-j} & 0 & 0 \\
					0 & 0 & 0 & I_j 	
	\end{matrix}
	\right),
\end{align}
where $0 \leq j \leq n$. Notice that $j=0$ corresponds to $w_{(n,n)}$ and $j=n$ corresponds to the identity matrix $e = I_{2n}$.  On the other hand, the elements of the set $W(\Omega^{\operatorname{ell}},\Omega^{\operatorname{ell}})$ satisfy $w\Omega^{\operatorname{ell}}=\Omega^{\operatorname{ell}}$ and thus normalize the block-diagonal Levi subgroup  $M_{(n,n)} = M_{\Omega^{\operatorname{ell}}}$.  One may quickly check that, of the elements of the form in \eqref{eq-nice-permutations-case-1}, only the identity $e$ and $w_{(n,n)}$ normalize $M_{(n,n)}$.
It follows that $[W_{\Omega^{\operatorname{ell}}} \backslash W_0 / W_{\Omega^{\operatorname{ell}}}] \cap W({\Omega^{\operatorname{ell}}},{\Omega^{\operatorname{ell}}})$ consists of precisely $e$ and $w_{(n,n)}$, proving the claim.

We now turn to studying the $M_\Omega^\theta$-distinction of $\mathcal F^y_\Omega(\tau,2)={}^y(\nu^{1/2}\tau\otimes\nu^{-1/2}\tau)$, where $y=ww_+$. 
There are two sub-cases to consider, either $w=e$ or $w=w_+w_{(n,n)}w_+^{-1}$.
If $w=e$, then  $y = w_+ \in [W_\Omega \backslash W_0 / W_\Omega]\cap W(\Omega,\Omega)$.   
As above, $\mathcal F^y_\Omega(\tau,2)={}^{w_+}(\nu^{1/2}\tau\otimes\nu^{-1/2}\tau)$.
If $w=w_{(n,n)}$, then $y=ww_+ = w_+w_{(n,n)}w_+^{-1}w_+ = w_+w_{(n,n)}$. 
It follows that $\mathcal F^y_\Omega(\tau,2)={}^{w_+w_{(n,n)}}(\nu^{1/2}\tau\otimes\nu^{-1/2}\tau)$.  Conjugation by $w_{(n,n)}$ interchanges the two $\mathrm{GL}$-blocks of $M_\Omega = M_{(n,n)}$; therefore, twisting a representation $\pi_1\otimes\pi_2$ of $M_\Omega$ by $w_{(n,n)}$ interchanges the two representations, that is, ${}^{w_{(n,n)}}(\pi_1\otimes\pi_2) \cong \pi_2\otimes\pi_1$.  Therefore, 
	${}^{w_+w_{(n,n)}}(\nu^{1/2}\tau\otimes\nu^{-1/2}\tau) = {}^{w_+}(\nu^{-1/2}\tau\otimes\nu^{1/2}\tau)$.
We have seen above that $M_\Omega^\theta \cong  \left( \Sp_n(F) \times \Sp_n(F) \right)$.
In both cases ($w=e, w = w_{(n,n)}$), it follows that $\mathcal{F}_\Omega^y(\tau,2)$ is $M_\Omega^\theta$-distinguished if and only if $\nu^{1/2}\tau$ and $\nu^{-1/2}\tau$ are $\Sp_n(F)$-distinguished.
By assumption, $\tau$ is an irreducible generic representation; therefore, by \cite[Theorem 3.2.2]{heumos--rallis1990}, $\Hom_{\Sp_n(F)}(\tau,1) = \{0\}$.  It follows, since $\nu$ is trivial on (maximal) unipotent subgroups of $\GL_n(F)$, that $\nu^s\tau$ is generic and $\Hom_{\Sp_n(F)}(\nu^{s}\tau,1) = \{0\}$, for every $s\in \C$.  Moreover, if $w$ is equal to either $e$ or $w_{(n,n)}$, then $\Hom_{M_\Omega^\theta}(\mathcal{F}_\Omega^y(\tau,2),1) = 0$, as claimed.
\end{proof}

\subsection{Case 2: `good' exponents}\label{sec-case-2}
Assume that $w\in [W_\Theta \backslash W_0 / W_\Omega]$ is such that $P_\Theta \cap {}^w M_\Omega$ is a proper parabolic subgroup of ${}^wM_\Omega$.
First, we show that $M_\Theta \cap {}^w P_\Omega$ is also a proper parabolic subgroup of $M_\Theta$.  We argue by contradiction, and suppose that $M_\Theta \cap {}^w P_\Omega = M_\Theta$.  By \cite[Proposition 1.3.3]{Casselman-book}, $M_\Theta \cap {}^w N_\Omega = \{e\}$ and $M_\Theta \cap {}^w M_\Omega = M_\Theta$.  In particular, $M_\Theta \subset {}^w M_\Omega = M_{w\Omega}$. However, both $M_\Omega$ and $M_\Theta$ are maximal Levi subgroups of $G$, and it follows that $M_\Theta = {}^wM_\Omega$.  This, in turn, implies that $P_\Theta \cap {}^wM_\Omega =M_\Theta = {}^w M_\Omega$ which contradicts our assumption that $P_\Theta \cap {}^w M_\Omega$ is a proper parabolic subgroup of ${}^wM_\Omega$.  We conclude that $M_\Theta \cap {}^w P_\Omega$ is a proper parabolic subgroup of $M_\Theta$. 

It follows from this last observation that if $y=ww_+$, then the representation $\mathcal F_N^y(\delta,2)=\iota_{M\cap {}^yQ}^M \left( {}^y (\nu^{1/2}\delta\otimes\nu^{-1/2}\delta)_{N \cap {}^yL} \right)$ is induced from ${}^y (\nu^{1/2}\delta\otimes\nu^{-1/2}\delta)_{N \cap {}^yL}$ along the proper parabolic $M\cap {}^yQ = M_\Theta \cap {}^wP_\Omega$ of $M=M_\Theta$; moreover, the Jacquet module ${}^y (\nu^{1/2}\delta\otimes\nu^{-1/2}\delta)_{N \cap {}^yL}$ is taken along the proper parabolic $P \cap {}^yL=P_\Theta \cap {}^wM_\Omega$ of ${}^y L={}^wM_\Omega$.
That is, both the Jacquet restriction and parabolic induction steps appearing in $\mathcal F_N^y(\delta,2)$ are along proper parabolic subgroups.
To be completely explicit, we note that
\begin{align*}
	\mathcal F_N^{ww_+}(\delta,2) = \iota_{M_\Theta \cap {}^wP_\Omega}^{M_\Theta} \left( {}^{ww_+} (\nu^{1/2}\delta\otimes\nu^{-1/2}\delta)_{N_\Theta \cap {}^wM_\Omega}\right).
\end{align*}
In this subsection, we will use the shorthand notation $[\tau]=\nu^{1/2}\tau\otimes \nu^{-1/2}\tau$, where $\tau$ is an irreducible admissible representation of $\GL_n(F)$.
Our goal is to compute the exponents of $\pi=\nu^{1/2}\delta\times\nu^{-1/2}\delta$ along $P=P_\Theta$; therefore, we need to understand the central characters of the irreducible subquotients of the $\mathcal F^{ww_+}(\delta,2)$.
By \cite[Lemma 4.16]{smith2018}, the quasi-characters appearing in $\Exp_{A_\Theta}(\mathcal F^{ww_+}(\delta,2))$ are the restrictions to $A_\Theta$ of the quasi-characters appearing in $\Exp_{A_{\Theta \cap w_\Omega}}({}^{ww_+}[\delta]_{N_\Theta\cap{}^wM_\Omega})$, where the $F$-split component of $M_\Theta \cap {}^w M_\Omega = M_{\Theta \cap w\Omega}$ is $A_{\Theta \cap w\Omega}$.  
Thus, our problem reduces to understanding the exponents of ${}^{ww_+}[\delta]$ along $P_\Theta\cap{}^wM_\Omega$.

Since $L = M_{(n,n)} \cong \GL_n(F) \times \GL_n(F)$, we have that $P_\Theta\cap{}^wM_\Omega \cong P_1 \times P_2$, where $P_1$ and $P_2$ are parabolic subgroups of $\GL_n(F)$, at least one of which is proper.
We can realize $w = w_+ w' w_+^{-1} \in [W_\Theta \backslash W_0 / W_\Omega]$, where $w' \in [W_{(2k,2n-2k)} \backslash W_0 / W_{(n,n)}]$.
Then, with $w = w_+ w' w_+^{-1}$,
\begin{align*}
 \left({}^{ww+} [\delta]\right)_{N_\Theta \cap {}^wM_\Omega} 
	& =  \left({}^{w_+ w'} [\delta]\right)_{w_+N_{(2k,2n-2k)}w_+^{-1} \cap w_+ w' M_{(n,n)}  w'^{-1} w_+^{-1}} \\	
	& =  \left({}^{w_+ w'} [\delta]\right)_{w_+\left(N_{(2k,2n-2k)} \cap w' M_{(n,n)}  w'^{-1}\right)w_+^{-1}} \\
	& =  {}^{w_+}\left( {}^{w'}[\delta]_{N_{(2k,2n-2k)} \cap w' M_{(n,n)}  w'^{-1}} \right)\\
	& =  {}^{w_+w'}\left( [\delta]_{w'^{-1}N_{(2k,2n-2k)} w'  \cap M_{(n,n)}} \right)\\
	& =  {}^{w_+w'}\left( [\delta]_{N_1\times N_2} \right),
\end{align*}
where we identify $P_\Theta\cap{}^wM_\Omega \cong P_1 \times P_2$ with a parabolic subgroup of $M_{(n,n)} \cong \GL_n(F) \times \GL_n(F)$ via
\begin{align*}
& P_1\times P_2  = w'^{-1}N_{(2k,2n-2k)} w'  \cap M_{(n,n)} \\
& = w'^{-1}\left( N_{(2k,2n-2k)}   \cap w'M_{(n,n)} w'^{-1} \right)w' \\
& = w'^{-1}w_+^{-1}\left( w_+N_{(2k,2n-2k)}w_+^{-1}   \cap w_+w'M_{(n,n)} w'^{-1}w_+^{-1} \right)w_+w' \\
& = w'^{-1}w_+^{-1}\left( w_+N_{(2k,2n-2k)}w_+^{-1}   \cap w_+w'w_+^{-1}w_+M_{(n,n)}w_+^{-1} w_+w'^{-1}w_+^{-1} \right)w_+w' \\
& = w'^{-1}w_+^{-1}\left( w_+N_{(2k,2n-2k)}w_+^{-1}   \cap ww_+M_{(n,n)}w_+^{-1} w^{-1} \right)w_+ w'\\
& = w'^{-1}w_+^{-1}\left( N_{\Theta}   \cap wM_{\Omega} w^{-1} \right)w_+w' \\
& = w_+^{-1}w^{-1}\left( N_{\Theta}   \cap wM_{\Omega} w^{-1} \right)w_+w,
\end{align*}
using that  $ww_+ = w_+w'w_+^{-1}w_+ = w_+w'$.
It follows that  
\begin{align*}
\left({}^{ww+} [\delta]\right)_{N_\Theta \cap {}^wM_\Omega} 
	& =  {}^{w_+w'}\left( [\delta]_{N_1\times N_2} \right)\\
	& =  {}^{w_+w'}\left( (\nu^{1/2}\delta \otimes \nu^{-1/2}\delta)_{N_1\times N_2} \right)\\
	& = {}^{w_+w'}\left( \nu^{1/2}\delta_{N_1} \otimes \nu^{-1/2}\delta_{N_2} \right)\\
	& = {}^{ww_+}\left( \nu^{1/2}\delta_{N_1} \otimes \nu^{-1/2}\delta_{N_2} \right),
\end{align*}
where in the final equality we have again used that $ww_+ = w_+w'$.
In the above calculation of $\left({}^{ww+} [\delta]\right)_{N_\Theta \cap {}^wM_\Omega}$, we also implicitly used the following basic fact.
\begin{lem}\label{lem-twist-jacuet-module}
Let $(\pi,V)$ be a smooth representation of $G=\GL_m(F)$.  Let $	P = MN$ be a (proper) parabolic subgroup of $G$ with Levi factor $M$ and unipotent radical $N$.  Let $s \in \C$.
Then the Jacquet module $(\nu^s \otimes \pi)_N$ is equivalent to the twisted Jacquet module $\nu^s\vert_{M}\otimes \pi_N$.
\end{lem}

\begin{proof}
The lemma follows immediately from the fact that $\nu$ is trivial on the unipotent group $N$.
Indeed, the space of both representations $\pi$ and $\nu^s\otimes \pi = \nu^s\pi$ is $V$.  The space of the Jacquet module of $\pi$, respectively $\nu^s\pi$, is the quotient of $V$ by the subspace
$V(N) = \spn\{ v - \pi(n)v : v\in V, n \in N \}$, respectively $\spn \{ v - \nu^s(n)\pi(n) v : v \in V, n \in N\}$.  Since $\nu^s(n) = 1$ for every $n\in N$, we see that the space of both $\pi_N$ and $(\nu^s\pi)_N$ is $V_N = V / V(N)$.
Finally, observe that for any $m\in M$ and $v + V(N) \in V_N$ we have
\begin{align*}
(\nu^s\pi)_N(m)(v + V(N)) & = \delta_P^{-1/2}(m)\nu^s(m)\pi(m) v + V(N) \\
 &= \nu^s(m)	\left(\delta_P^{-1/2}(m)\pi(m) v + V(N)\right) \\
 &=  \nu^s(m) \pi_N(m) (v + V(N));
\end{align*}
therefore $(\nu^s \pi)_N = \nu^s\vert_M \otimes \pi_N$, as claimed.
\end{proof}

In order to understand the exponents of $\left({}^{ww+} [\delta]\right)_{N_\Theta \cap {}^wM_\Omega}$, we require the following proposition.

\begin{prop}\label{prop-exponents-of-tensor-product}
Let $G$ and $G'$ be two connected reductive groups over $F$.
Let $(\pi,V)$, respectively $(\sigma, W)$, be a finitely generated admissible representation of $G$, respectively $G'$.
The set of exponents of the (external) tensor product $\pi\otimes \sigma$ consists of all pairwise products $\chi \otimes \chi'$, where $\chi \in \Exp_{Z_G}(\pi)$ and $\chi ' \in \Exp_{Z_{G'}}(\sigma)$ are exponents of $\pi$ and $\sigma$ respectively.  That is,
\begin{align}\label{eq-exponents-of-tensor-product}
\Exp_{Z_G \times Z_{G'}}(\pi \otimes \sigma) 
& = \{ \chi \otimes \chi' : \chi \in \Exp_{Z_G}(\pi), \chi ' \in \Exp_{Z_{G'}}(\sigma) \}\\
\nonumber & \cong \Exp_{Z_G}(\pi) \times \Exp_{Z_{G'}}(\sigma).
\end{align}
\end{prop}

\begin{proof}
The exponents $\Exp_{Z_G \times Z_{G'}}(\pi \otimes \sigma)$ of $\pi \otimes \sigma$ are precisely the central characters of the irreducible subquotients of $\pi \otimes \sigma$ (\textit{cf.}~\cite[Proposition 2.1.9]{Casselman-book}, \cite[Lemma 4.14]{smith2018}).
To prove the proposition, it is sufficient to show that the irreducible subquotients of $\pi \otimes \sigma$ are of the form $V^j \otimes W^k$, where $V^j$, respectively $W^k$, is an irreducible subquotient of $(\pi,V)$, respectively $(\sigma,W)$.  Indeed, if $V^j$ (resp.~$W^k$) is irreducible, then it admits a central character $\chi_j$ (resp.~$\chi_k$); moreover, $V^j \otimes W^k$ has central character $\chi_j \otimes \chi_k : Z_G \times Z_{G'} \rightarrow \C^\times$.
We omit the proof of the elementary fact regarding the subquotients of the external tensor product $(\pi\otimes \sigma, V\otimes W)$.
\end{proof}

\begin{note}
To clarify the following calculations we introduce some additional notation for certain subsets of $\Delta$.
For and $1 \leq j \leq 2n-1$, let $\Xi_j = \Delta\setminus\{\epsilon_{j}-\epsilon_{j+1}\}$.
We will be particularly interested in $\Xi_{2k}$ and $\Xi_n=\Omega^{\operatorname{ell}}$ since $\Theta = w_+ \Xi_{2k}$ and $\Omega = w_+ \Xi_n$.
\end{note}

Recall that the $(\theta,F)$-split component $S_\Theta$ of $M_\Theta$ is equal to the $F$-split component $A_\Theta$.  In particular, the $(\theta,F)$-split component of $G$ is $S_G = A_G$. We now consider the exponents of $\left({}^{ww+} [\delta]\right)_{N_\Theta \cap {}^wM_\Omega} = {}^{ww_+}\left( \nu^{1/2}\delta_{N_1} \otimes \nu^{-1/2}\delta_{N_2} \right)$ restricted to $S_\Theta^- \setminus S_\Theta^1S_G = A_\Theta^- \setminus A_\Theta^1A_G$.
Let $s\in S_\Theta = A_\Theta$.  Since $A_\Theta = w_+ A_{(2k, 2n-2k)} w_+^{-1}$, we can write $s = w_+ a w_+^{-1}$, where $a = \diag(a_1 1_{2k}, a_2 1_{2n-2k})$ lies in $A_{(2k, 2n-2k)}^- \setminus A_{(2k, 2n-2k)}^1A_G$.
In particular, $A_{(2k, 2n-2k)} = A_{\Xi_{2k}}$ and $a$ has the property that $|\epsilon_{2k}-\epsilon_{2k+1}(a)| = |a_1 a_2^{-1}| < 1$.
By \Cref{prop-exponents-of-tensor-product} and \Cref{lem-twist-jacuet-module}, the exponents of $\left({}^{ww+} [\delta]\right)_{N_\Theta \cap {}^wM_\Omega} = {}^{ww_+}\left( \nu^{1/2}\delta_{N_1} \otimes \nu^{-1/2}\delta_{N_2} \right)$ are all of the form ${}^{ww_+}\left(\nu^{1/2}\chi_1 \otimes \nu^{-1/2}\chi_2 \right)$, where $\chi_1 \in \Exp_{A_1}(\delta_{N_1})$, and $\chi_2 \in \Exp_{A_2}(\delta_{N_2})$.
Here we write $A_i$ for the $F$-split component of $M_i \subset P_i \subset \GL_n(F)$, $i=1,2$.
In particular,
\begin{align*}
	{}^{ww_+}\left(\nu^{1/2}\chi_1 \otimes \nu^{-1/2}\chi_2 \right) (s)
	& = {}^{ww_+}\left(\nu^{1/2}\chi_1 \otimes \nu^{-1/2}\chi_2 \right) (w_+ a w_+^{-1}) \\
	& = \left(\nu^{1/2}\chi_1 \otimes \nu^{-1/2}\chi_2 \right) (w_+^{-1}w^{-1}w_+ a w_+^{-1}ww_+) \\
	& = \left(\nu^{1/2}\chi_1 \otimes \nu^{-1/2}\chi_2 \right) ({w'}^{-1} a w'),
\end{align*}
where $w' = w_+^{-1} w w_+ \in [W_{(2k,2n-2k)} \backslash W_0 / W_{(n,n)}]$ and 
\begin{align*}
	{w'}^{-1} a w' & \in {w'}^{-1} A_{(2k,2n-2k)}^- {w'} \setminus {w'}^{-1} A_{(2k,2n-2k)}^1 {w'}A_G \\
	& \subset A_{{w'}^{-1}M_{(2k,2n-2k)}w' \cap M_{(n,n)}}^- \setminus A_{{w'}^{-1}M_{(2k,2n-2k)}w' \cap M_{(n,n)}}^1 A_{(n,n)}\\
	& = A_{({w'}^{-1}\Xi_{2k}) \cap  \Xi_{n}}^- \setminus A_{({w'}^{-1}\Xi_{2k}) \cap  \Xi_{n}}^1 A_{(n,n)} \\
	&= A_{M_1\times M_2}^- \setminus A_{M_1\times M_2}^1A_{(n,n)}\\
	& = A_1^- \times A_2^-\setminus (A_1^1 \times A_2^1) A_{(n,n)}.
\end{align*}
Where the containment in the second line follows as in the proof of \cite[Lemma 8.4]{smith2018}.
By assumption, $\delta$ is a discrete series representation of $\GL_n(F)$; therefore, the exponents $\chi_1$ and $\chi_2$ of $\delta$ satisfy Casselman's Criterion (\cite[Theorem 6.5.1]{Casselman-book}) and
\begin{align*}
|\chi_1\otimes \chi_2 (	{w'}^{-1} a w') | & < 1.
\end{align*}
To ensure that the exponents ${}^{ww_+}\left(\nu^{1/2}\chi_1 \otimes \nu^{-1/2}\chi_2 \right)$ of $\left({}^{ww+} [\delta]\right)_{N_\Theta \cap {}^wM_\Omega} = {}^{ww_+}\left( \nu^{1/2}\delta_{N_1} \otimes \nu^{-1/2}\delta_{N_2} \right)$ satisfy the Relative Casselman's Criterion (\cite[Theorem 4.7]{kato--takano2010}), we need to ensure that
\begin{align}\label{eq-unramified-relative-cass}
|\nu^{1/2}\otimes \nu^{-1/2} (	{w'}^{-1} a w') | & \leq 1.
\end{align}
We can realize the restriction of the unramified character $\nu^{1/2} \otimes \nu^{-1/2}$  to the maximal (diagonal) $F$-split torus $A_0$ as the composition of $|\cdot |_F^{1/2}$ with the sum over all roots in $\Delta$ with positive integral coefficients, that is, 
\begin{align}\label{eq-unramified-char-roots-2}
	(\nu^{1/2} \otimes \nu^{-1/2}) \vert_{A_0} & = |\cdot|_F^{1/2} \circ \left( \sum_{\alpha \in \Delta} c_\alpha \cdot \alpha \right),
\end{align}
where $c_{\epsilon_{i}-\epsilon_{i+1}} = i$, for $1\leq i \leq n$, and $c_{\epsilon_{n+j}-\epsilon_{n+j+1}} = n-j$, for $1\leq j \leq n-1$. 
To compute $(\nu^{1/2} \otimes \nu^{-1/2})(w'^{-1}aw')$ it is helpful to partition $\Delta$ as the disjoint union of $({w'}^{-1}\Xi_{2k}) \cap  \Xi_{n}$ and $\Delta\setminus (({w'}^{-1}\Xi_{2k}) \cap  \Xi_{n})$.
Indeed, 
since $A_{w'^{-1}\Xi_{2k}} \subset A_{(w'^{-1}\Xi_{2k})\cap\Xi_n}$, it follows that 
$\alpha(w'^{-1}aw')=1$, for all $\alpha \in ({w'}^{-1}\Xi_{2k}) \cap  \Xi_{n}$.
On the other hand, 
since $w'^{-1}aw' \in A_{(w'^{-1}\Xi_{2k})\cap\Xi_n}^-$ we have that $|\beta(w'^{-1}aw')|_F\leq 1$, for all $\beta \in \Delta\setminus (({w'}^{-1}\Xi_{2k}) \cap  \Xi_{n})$.
From \eqref{eq-unramified-char-roots-2}, it follows that
\begin{align*}
	& (\nu^{1/2} \otimes \nu^{-1/2}) (w'^{-1}aw')  = \prod_{\alpha \in \Delta} |\alpha(w'^{-1}aw')|^{c_\alpha/2}_F \\
	& = \left(\prod_{\alpha \in ({w'}^{-1}\Xi_{2k}) \cap  \Xi_{n}} |\alpha(w'^{-1}aw')|^{c_\alpha/2}_F\right) \left(  \prod_{\beta \in \Delta\setminus (({w'}^{-1}\Xi_{2k}) \cap  \Xi_{n})} |\beta(w'^{-1}aw')|^{c_\beta/2}_F\right) \\
	& = \prod_{\beta \in \Delta\setminus (({w'}^{-1}\Xi_{2k}) \cap  \Xi_{n})} |\beta(w'^{-1}aw')|^{c_\beta/2}_F\\
	& \leq 1,
\end{align*}
which establishes the truth of \eqref{eq-unramified-relative-cass}.
Moreover, we now have that
\begin{align*}
\left\vert{}^{ww_+}\left(\nu^{1/2}\chi_1 \otimes \nu^{-1/2}\chi_2 \right)(s)\right\vert
& = \left(\nu^{1/2}\chi_1 \otimes \nu^{-1/2}\chi_2 \right) ({w'}^{-1} a w')\\
& = |\chi_1\otimes \chi_2 (	{w'}^{-1} a w') ||\nu^{1/2}\otimes \nu^{-1/2} (	{w'}^{-1} a w') | \\
& < 1,
\end{align*}
for all $\chi_1 \in \Exp_{A_1}(\delta_{N_1})$, $\chi_2 \in \Exp_{A_2}(\delta_{N_2})$, and $s = w_+aw_+^{-1} \in S_\Theta\setminus S_\Theta^1S_G$, where $w' = w_+^{-1}ww_+$ as above.
Finally, we have established the desired result:
\begin{prop}\label{prop-good-exp}
	Let $\Theta = \Theta_k$, $1\leq k \leq n-1$ be a maximal $\theta$-split subset of $\Delta_0$.
	Let $w\in [W_\Theta \backslash W_0 / W_\Omega]$ be such that $P_\Theta \cap {}^w M_\Omega$ is a proper parabolic subgroup of ${}^wM_\Omega$.
	Let $\delta$ be an irreducible admissible square integrable representation of $\GL_n(F)$.
	The exponents of $\pi_N$, and $\style{U}(\delta,2)_N$, corresponding to the irreducible subquotients of $\mathcal{F}_N^{ww_+}(\delta,2) = \left({}^{ww+} [\delta]\right)_{N_\Theta \cap {}^wM_\Omega}$ satisfy the condition \eqref{eq-rel-cass-condition} of \Cref{thm-relative-casselman}.
\end{prop}

\bibliographystyle{amsalpha}

\bibliography{jerrod-refs}
%

\end{document}